\documentclass[10]{amsart}
\usepackage{amsbsy,amssymb,pstricks,pst-node,mathrsfs,MnSymbol}
\usepackage{subfigure}
\usepackage[utf8]{inputenc}
\usepackage[left]{lineno}
\usepackage{tikz}
\usetikzlibrary{positioning}
\theoremstyle{definition}

\newtheorem{theorem}{Theorem}[section]
\newtheorem{lemma}[theorem]{Lemma}
\newtheorem{remark}[theorem]{Remark}
\newtheorem{proposition}[theorem]{Proposition}
\newtheorem{corollary}[theorem]{Corollary}

\newtheorem{observation}[theorem]{Observation}
\newtheorem{definition}[theorem]{Definition}

\definecolor{aquamarine}{rgb}{0.5, 1.0, 0.83}

\begin{document}
	

	%
	\title{Coloring of zero-divisor graphs of posets and applications to graphs associated with algebraic structures
}\maketitle
\markboth{ Nilesh Khandekar and Vinayak Joshi}{ Coloring of Zero-divisor graphs and Applications }\begin{center}\begin{large} Nilesh Khandekar and Vinayak
		Joshi\end{large}\\\begin{small}\vskip.1in\emph{Department of Mathematics,
			Savitribai Phule Pune University,\\ Pune - 411007, Maharashtra,
			India}\\E-mail: \texttt{khandekarnilesh11@gmail.com, vvjoshi@unipune.ac.in, vinayakjoshi111@yahoo.com}\end{small}\end{center}\vskip.2in
\begin{abstract} In this paper, we characterize chordal and perfect zero-divisor graphs of finite posets. Also, it is proved that the zero-divisor graphs of finite posets and the complement of zero-divisor graphs of finite $0$-distributive posets satisfy the Total Coloring Conjecture. These results are applied to the zero-divisor graphs of finite reduced rings, the comaximal ideal graph of rings, the annihilating ideal graphs, the intersection graphs of ideals of rings, and the intersection graphs of subgroups of cyclic groups. In fact, it is proved that these graphs associated with a commutative ring $R$ with identity can be effectively studied via the zero-divisor graph of a specially constructed poset from $R$. 
\end{abstract}\vskip.2in
\noindent\begin{Small}\textbf{Mathematics Subject Classification (2020)}:
	05C15, 05C17, 05C25, 06E05, 06E40, 06D15, 06A07, 13A70     \end{Small}\vskip.2in
\vskip.25in

\baselineskip 14truept 
\section{Introduction}\label{intro}
The study of graphs associated with algebraic and ordered structures is an active area of research. Some of the interesting classes of such graphs are the zero-divisor graphs of rings, comaximal ideal graphs of rings, annihilating ideal graphs of rings, and intersection graphs of ideals of rings. Besides being interesting in their own right, these classes of graphs, on the one hand, have served as a testing ground for some of the conjectures in graph theory, while on the other hand demonstrate the rich interplay that comes with associating graphs to algebraic and ordered structures.

One of the important aspects of graph theory is the notion of the coloring of graphs and the computation of a graph's chromatic number. It is well known that the coloring of graphs is an NP-Complete problem; see \cite{Kar72}. However, for perfect graphs, coloring can be done in polynomial time. A \textit{perfect graph} is a finite simple graph in which the chromatic number of every induced subgraph is equal to its clique number, that is, the order of the subgraph's largest clique. Efforts have been made to determine the classes of perfect graphs. For instance, it is known that the class of chordal graphs is perfect; see Dirac \cite{dirac}. The notion of perfectness, weakly perfectness and chordalness of graphs associated with algebraic structures has been an active area of research; see \cite{aghapouramin},  \cite{azadi}, \cite{bagheri}, \cite{Be},  \cite{das}, \cite{mirghadim}, \cite{ashkan}, \cite{smith},  etc.

On the lines of the zero-divisor graphs of rings, we study the zero-divisor graphs of ordered sets to construct examples of perfect graphs and also, at the same time, demonstrate the rich interplay that naturally comes with studying the questions of coloring. However, we would like to state that not much attention has been given to the interplay of the zero-divisor graphs of ordered sets and the graphs associated with algebraic structures. In this paper, we aim to mainly explore these aspects in the last section and highlight the interplay.

Besides the vertex and edge colorings of graphs, total coloring has been another important coloring. The \textit{total coloring} of a graph $G$ is an assignment of colors to the vertices and the edges of $G$ such that every pair of adjacent vertices, every pair of incident edges, and every vertex and incident edge pair receive different colors. The total chromatic number $\chi" (G)$ of a graph $G$ is the minimum number of colors needed in a total coloring of $G$; see \cite{vgv1}.
Vizing \cite{vgv1, vgv2}  and  Behzad \cite{mbehzad} studied the total coloring of graphs. They both formulated the following conjecture, now known as the Total Coloring Conjecture.

\vskip 5truept 
\noindent \textbf{Total Coloring Conjecture:} {\it Let $G$ be a finite simple undirected graph. Then $\chi''(G) =\Delta(G)+1$ or $\chi''(G)=\Delta(G)+2$.}
\vskip 5truept

One of our main establishes the Total Coloring conjecture for graphs associated with finite posets and the complement of zero-divisor graphs of finite 0-distributive posets. As a consequence of these results, we are able to prove that the comaximal graphs and the complement of the comaximal graphs of commutative rings satisfy the Total Coloring Conjecture.

We now briefly discuss the contents of each Section. Section 2 deals with the preliminary results from ordered sets needed for the paper's later development. Sections 3 and 4 deal with the chordal zero-divisor graphs and perfect zero-divisor graphs of ordered sets. In section 5, we prove one of our main results that establish the Total Coloring conjecture for zero-divisor graphs finite posets. The work in this section is a continuation of the earlier work in \cite{nkvj}. Section 6 focuses on applications of these ideas to the study of the interplay between zero-divisor graphs ordered sets and various graphs associated with algebraic structures such as the divisor graphs of rings, the comaximal ideal graphs of rings, the annihilating ideal graphs of rings, the intersection graphs of ideals of rings.

In fact, we prove the following main results and their corollaries. One can refer to the subsequent sections for the terminologies and notations mentioned in the following results.

\begin{theorem} \label{zdgchordal}
	Let $P$ be a finite poset such that $[P]$ is a Boolean lattice. Then
	
	\textbf{(A)} $G(P)$ is chordal if and only if one of the following hold:
	
	\begin{enumerate}
		\item $P$ has exactly one atom;
		
		\item $P$ has exactly two atoms with $|P_i|=1$ for some $i\in \{1,2\}$;

		\item $P$ has exactly three atoms with $|P_i|=1$ for all $i\in \{1,2,3\}$.
	\end{enumerate}
	
	\textbf{(B)} $G^c(P)$ is chordal if and only if   number of atoms of $P$ are at most $3$.

	\textbf{(C)}	$G(P)$
	is perfect if and only if  
	$P$ has at most 4 atoms.
\end{theorem}

\begin{theorem}\label{zdgtcc}
	Let $P$ be a finite poset. Then $G(P)$ satisfies the Total Coloring Conjecture. Moreover, if $P$ is a finite 0-distributive poset, then $G^c(P)$ satisfies Total Coloring Conjecture.
\end{theorem}

\begin{corollary} \label{cgchordperfect}
	Let $R= R_1 \times R_2 \times \cdots \times R_n$ be a ring with identity such that each $R_i$ is a local ring with   finitely many ideals. Let $\mathbb{CG}^*(R)$ be the comaximal graph, $\mathbb{CAG}^*(R)$ be the co-annihilating ideal graph and $\mathbb{AG}^{*c}(R)$ be the complement of the annihilating ideal graph of $R$.  Then
	
	\textbf{(A)}  $\mathbb{CG}^*(R)=\mathbb{CAG}^*(R)=\mathbb{AG}^{*c}(R)$.
	
	\textbf{(B)}  $\mathbb{CG}^*(R)=\mathbb{CAG}^*(R)=\mathbb{AG}^{*c}(R)$
	is chordal if and only if  one of the following hold:
	
	\begin{enumerate}
		\item $n=1$;
		
		\item  $n=2$  and $R_i$ is field for some $i\in \{1,2\}$;

		\item $n=3$ and $R_i$ is field for all $i\in \{1,2,3\}$.
	\end{enumerate}
	
	\textbf{(C)} $\mathbb{CG}^{*c}(R)=\mathbb{CAG}^{*c}(R)=\mathbb{AG}^{*}(R)$ is chordal if and only if  $n \leq 3$.	
	
	\textbf{(D)}	$\mathbb{CG}^*(R)=\mathbb{CAG}^*(R)=\mathbb{AG}^{*c}(R)$ is perfect if and only if $n \leq 4$.
	
	\textbf{(E)} $\mathbb{CG}^*(R)$ and $\mathbb{CG}^{*c}(R)$ satisfies the Total Coloring Conjecture. Moreover, the edge chromatic number  $\chi'(\mathbb{CG}^*(R))=\Delta(\mathbb{CG}^*(R))$. 
	
	\textbf{(F)} Moreover, if each $R_i$ is a local  Artinian principal ideal ring, then the statements $(B), (C), (D)$ and $(E)$ are also true, if we replace $\mathbb{CG}^*(R)$ by the complement of the  intersection graph $\mathbb{IG}^c(R)$ of ideals  of $R$.
	
\end{corollary}

\section{Preliminaries}
\par We begin with the following necessary definitions and terminologies given in  Devhare et al. \cite{djl}. Also, we prove  the key results (cf. Lemma \ref{property}) that are required to prove chordality, perfectness and Total Coloring Conjecture.  
\vskip 5truept 
\begin{definition}[Devhare et al. \cite{djl}]
	\par Let $P$ be a poset. Given any $ A\subseteq P$, the \textit{upper cone} of $A$ is given by 
	$A^u=\{b\in P$ $|$ $b\geq a$ for every $a\in A\}$ and the \textit{lower cone} of $A$ is given by $A^\ell=\{b\in P$ $|$
	$b\leq a$ for every $a\in A\}$. If $a\in P$, then the sets $\{a\}^u$ and
	$\{a\}^\ell$ will be denoted by $a^u$ and $a^\ell$, respectively. By $A^{u\ell}$, we mean $\{A^u\}^\ell$. Dually, we have the notion of $A^{\ell u}$.

	A poset $P$ with  0 is called \textit{$0$-distributive} if  $\{a,b\}^\ell=\{0\}=\{a,c\}^\ell$ implies $\{a,\{b,c\}^u\}^\ell=0$; see \cite{jw1}. Note that if $\{b,c\}^u=\emptyset$, then $\{b,c\}^{u\ell}=P$. 
	A lattice $L$ with $0$ is said to be {\it $0$-distributive }if $a\wedge b=0$ and $a\wedge c=0$ implies $a\wedge(b\vee c)=0$. Hence it is clear that if a lattice $L$ is 0-distributive, then $L$, as a poset, is a 0-distributive poset. Dually, a  lattice $L$ with $1$ is said to be \textit{$1$-distributive} if $a\vee b=1$ and $a\vee c=1$ implies $a\vee(b\wedge c)=1$. A lattice $L$ is \textit{modular} if  for all $a,b,c \in L$, $a \leq b$ implies  \mbox{$(a\vee c)\wedge b=a\vee(c\wedge b)$; see \cite[page 132]{stern}.}

	\par Suppose that $P$ is a poset with  $0$. If $\emptyset\neq A \subseteq P$, then
	the {\it annihilator} of $A$ is given by \linebreak
	$A^{\perp} = \{ b \in P ~~\vert ~~\{a, b\}^\ell = \{0\}~~ {\rm for~
		all }~~ a \in A \}$, and if $A = \{a\}$, then we write $a^\perp=A^\perp$. An element
	$a\in P$  is an {\it atom} if $a >0$ and $\{b\in P$ $|$ $0<b<a\}=\emptyset$, and
	$P$ is called \emph{atomic} if for every $b\in P\setminus\{0\}$, there exists an
	atom $a\in P$ such that $a\leq b$. A poset $P$ with 0 is said to {\it section semi-complemented} (in brief SSC poset), if for $a \not \leq b$, there exists a nonzero element $c \in P$ such that $ c\leq a$ and $\{b,c\}^\ell =\{0\}$. An atomic, SSC poset is called an {\it atomistic} poset. Equivalently, $P$ is atomistic, if for $a\not\leq b$, then there is an atom $p\in P$ such that $p \leq a$ and $p \not \leq b$; see Joshi \cite{vjssc}. 
	\par  A poset $P$ is called \emph{bounded} if $P$ has both the least element $0$
	and the greatest element $1$. An element $a'$ of a bounded poset $P$ is a
	\emph{complement} of $a\in P$ if $\{a,a'\}^\ell=\{0\}$ and
	$\{a,a'\}^u=\{1\}$. A \textit{pseudocomplement} of $a\in P$ is an element
	$b\in P$ such that $\{a,b\}^\ell=\{0\}$, and $x\leq b$ for every $x\in P$ 
	with $\{a,x\}^\ell=\{0\}$; that is, $b$ is the pseudocomplement of $a$ if and
	only if $a^\perp=b^\ell$; see Venkatanarasimhan \cite{venkat} (see also Hala\v{s} \cite{halas}). It is straightforward to check that any element
	$a\in P$ has at most one pseudocomplement, and it will be denoted by $a^*$ (if
	it exists). A bounded poset $P$ is called \emph{complemented} (respectively,
	\textit{pseudocomplemented}) if every element $a$ of $P$ has a complement $a'$
	(the pseudocomplement $a^*$).

	\par A  poset $P$ is called \emph{distributive} if,  
	$\{\{a\}\cup\{b,c\}^u\}^\ell=\{\{a,b\}^\ell\cup\{a,c\}^\ell\}^{u\ell}$
	holds for all
	$a,b,c\in P$; see \cite{LR}.  This definition generalizes the usual notion of
	a distributive lattice (i.e., a  lattice is distributive in the usual
	sense if and only if it is a distributive poset). Moreover, a bounded poset $P$
	is called \emph{Boolean}, if $P$ is distributive and complemented; see \cite{H}. Clearly, every
	Boolean algebra is a Boolean poset, but the converse is not true.

	\par It is well-known that in a Boolean poset, complementation is nothing but the pseudocomplementation (cf. \cite[Lemma 2.4]{JK}). In particular, if $P$ is Boolean,
	then $P$ is pseudocomplemented, and every element $x\in P$ has the unique    complement $x'$. Further, every Boolean poset is atomistic (cf. \cite{jw}).
	\vskip 2truept 
	
	The concept of zero-divisor graph of a poset is introduced in \cite{HJ}, which was modified later in \cite{LW}.
	\vskip 2truept 
	
	\par Let $P$ be a poset with $0$. Define a \emph{zero-divisor} of $P$ to be any element of the set \linebreak $Z(P)=\{a\in P$ $|$ there exists $b\in P\setminus\{0\}$ such that $\{a,b\}^\ell=\{0\}\}$. As in \cite{LW}, the \emph{zero-divisor graph} of $P$ is the graph $G(P)$ whose vertices are the elements of $Z(P)\setminus\{0\}$ such that two vertices $a$ and $b$ are adjacent if and only if $\{a,b\}^\ell=\{0\}$. If $Z(P)\neq\{0\}$, then clearly $G(P)$ has at least two vertices, and $G(P)$ is connected with diameter at most three (\cite[Proposition 2.1]{LW}). We abuse the notation  $G^*(P)$ for the \emph{zero-divisor graph} of $P$ with the vertex set  $P\setminus\{0,1\}$, if $P$ has the greatest element 1, if $P$ do not have 1, then the vertex set is $P\setminus\{0\}$. Further, two vertices $a$ and $b$ are adjacent in $G^*(P)$ if and only if $\{a,b\}^\ell=\{0\}$. So from the notation $G(P)$ or $G^*(P)$, the underlined vertex set of the zero-divisor graph is clear and the adjacency relation remains same in both the graphs.
	\par We set    $\mathcal{D}=P\setminus Z(P)$. The elements  $d\in\mathcal{D}$ are the {\it dense elements} of $P$.
	
	\vskip 2truept 
	\par  For a poset $P$ with $0$, an equivalence relation
	$\sim$ is given on $P$ by $a \sim b$ if and only if $a^\perp = b^\perp$. The set
	of equivalence classes of $P$ will be denoted by $[P^\sim]=\bigl\{[a]~ \vert ~~a \in
	P~~\bigr\}$, where $[a] = \{x\in P ~\vert~ x \sim a\}$. Clearly, $[0]=\{0\}$, and if
	$d\in\mathcal{D}$ then $[d]=\mathcal{D}$; see Joshi et al. \cite{jwp} and  Devhare, Joshi and LaGrange \cite{djl}.
	
	\vskip 2truept 
	\par Note that $[P^\sim]$ is a poset under a partial order given by $[a] \leq [b]$
	if and only if $b^\perp\subseteq a^\perp$. From the observation that
	$b^\perp\subseteq a^\perp$ whenever $a,b\in P$ with $a\leq b$, it follows that
	the canonical mapping $P\rightarrow [P^\sim]$ defined by $a\mapsto[a]$ is an
	order-preserving surjection. Furthermore, if $a$ is an atom of the poset $P$
	then, for every $b\in P\setminus\{0\}$, either $a\leq b$, or $\{a,b\}^\ell=\{0\}$. It follows
	that if $a$ is an atom of $P$, then $[a]$ is an atom of $[P^\sim]$; however, the converse is
	not true.
	\par Let $P$ be a pseudocomplemented poset. If $a,b\in P$, then $a^*\leq b^*$ if
	and only if $a^*\in b^\perp$, if and only if $a^\perp=(a^*)^\ell\subseteq
	b^\perp$. That is, $a^*\leq b^*$ if and only if $[b]\leq[a]$ and, in particular,
	$a^*=b^*$ if and only if $[a]=[b]$. 

	\par Note that pseudocomplemented poset need not be Boolean.

	\vskip 5truept

	\vskip 5truept

	\par The {\it direct product} of posets $P^1,\dots,P^n$ is the poset
	$\textbf{P}=\prod\limits_{i=1}^nP^i$ with $\leq$ defined such that $a\leq b$ in $\textbf{P}$ if and only if
	$a_{i}\leq b_{i}$ (in $P^i$) for every $i\in\{1,\dots,n\}$. For any
	$\emptyset\neq A\subseteq \prod\limits_{i=1}^nP^i$, note that
	$A^u=\{b\in\prod\limits_{i=1}^nP^i$ $|$ $b_{i}\geq a_{i}$ for every $a\in A$
	and $i\in\{1,\dots,n\}\}$. Similarly, $A^\ell=\{b\in\prod\limits_{i=1}^nP^i$
	$|$ $b_{i}\leq a_{i}$ for every $a\in A$ and $i\in\{1,\dots,n\}\}$. Clearly,  $\textbf{P}$ is a  $0$-distributive poset, if $Z(P^i)=\{0\}$ for every $i$. We set $A_{{q}_{_1}} = ({q}_{_1}^u)\setminus \mathcal{D}$ and define $A_{{q}_{_j}} = \Big({{q}_{_j}}^{u}\Big)\setminus\Big(\mathcal{D}\cup (\bigcup \limits_{i=1}^{j-1}{{q}_{_{i}}}^{u})\Big)$ for every $j\in\{2,\dots,n\}$.

	\vskip 10truept

	\par  Throughout, $P$ denotes a poset with $0$ and $q_i$, $i \in\{1,2, \cdots, n\}$ are all atoms of $P$, where $n \geq 2$. All graphs are finite simple graphs. The poset $ \mathbf{ P}$ is  $\prod\limits_{i=1}^{n}P^i$,  where $P^i$'s are  finite bounded posets such that $Z(P^i)=\{0\}$ and $2\leq|P^i|$, $\forall i$. 
	\vskip 10truept 
	
	Afkhami et al. \cite{afkhami} partitioned the set $P\setminus \{0\}$ as follows.
	
	\par	Let $1\leq i_1< i_2<\dots<i_k\leq n$. The notation $P_{i_1i_2\dots i_k}$ stands for the set: $$P_{i_1i_2\dots i_k}=\Bigg\{x\in P~\mathbin{\Big|}~x\in\biggl( \bigcap\limits_{s=1}^k\{q_{_{i_s}}\}^u\biggr) \mathbin{\Big\backslash}\biggl(\bigcup\limits_{j\neq i_1,i_2,\dots,i_k}\{q_{_j}\}^u\biggr)\Bigg\}. \hfill{     \hspace{.2in}  -------(\circledcirc)}$$

	In \cite{afkhami}, the following observations are proved.
	\begin{enumerate}
		
		\item If the index sets $\{i_1,\dots,i_k\}$ and $\{j_1,\dots,j_{k'}\}$ of $P_{i_1i_2\dots i_k}$ and $P_{j_1j_2\dots j_{k'}}$, respectively, are distinct, that is, $\{i_1,\dots,i_k\}\neq \{j_1,\dots,j_{k'}\}$, then $(P_{i_1i_2\dots i_k})\cap (P_{j_1j_2\dots j_k'})=\emptyset$.
		\item $\displaystyle P\backslash \{0\}= \bigcupdot\limits_{\substack{k=1,\\ 1\leq i_1< i_2<\dots<i_k\leq n}}^{n} P_{i_1i_2\dots i_k}$.
	\end{enumerate}

	\par Define a relation $\approx$ on $P\setminus\{0\}$ as follows:
	$x\approx y$ if and only if $x,y\in P_{i_1i_2\dots i_k}$ for some partition $P_{i_1i_2\dots i_k}$ of $P\setminus \{0\}$.
	It is easy to observe that $\approx$ is an equivalence relation.
	The following result proves that the equivalence relations $\sim$ and $\approx$ are the same.
	
\end{definition}
\vskip 5truept

\begin{lemma}\label{eqsame} The equivalence relations
	$\sim$ and $\approx$ are same on $P\setminus \{0\}$.
\end{lemma}

\begin{proof}
	Let $t\in P_{i_1i_2\dots i_k}=\Bigg\{x\in P~\mathbin{\Big|}~x\in\biggl( \bigcap\limits_{s=1}^k\{q_{_{i_s}}\}^u\biggr) \mathbin{\Big\backslash}\biggl(\bigcup\limits_{j\neq i_1,i_2,\dots,i_k}\{q_{_j}\}^u\biggr)\Bigg\}$. Then it is easy check that, \linebreak $t^\perp =\biggl(\bigcup\limits_{j\neq i_1,i_2,\dots,i_k}\bigl\{q_{_j}\bigr\}^u\biggr)\mathbin{\Big\backslash} \biggl(\bigcup\limits_{s=1}^k\bigl\{q_{_{i_s}}\bigr\}^u\biggr)$. From this, it is clear that, if $x,y\in P_{i_1i_2\dots i_k}$, then $x^\perp=y^\perp$. Thus $x\approx y$ implies that $x\sim y$. 	Conversely, assume that  $x\sim y$. Then $x^\perp=y^\perp$. We have to prove that $x\approx y$, that is, $x,y\in P_{i_1i_2\dots i_k}$ for some partion $P_{i_1i_2\dots i_k}$ of $P\setminus \{0\}$. Suppose on contrary, there exist the index sets $\{i_1,\dots,i_k\}$ and $\{j_1,\dots,j_{k'}\}$ of $P_{i_1i_2\dots i_k}$ and $P_{j_1j_2\dots j_{k'}}$, respectively, that are distinct such that $x\in P_{i_1i_2\dots i_k}$ and $y\in P_{j_1j_2\dots j_{k'}}$. Clearly, $x\neq y$. Since $\{i_1,\dots,i_k\}$ and $\{j_1,\dots,j_{k'}\}$ are distinct, then $i_p\notin \{j_1,\dots,j_{k'}\}$ for some $p\in \{1,\dots,k\}$ or $j_q\notin  \{i_1,\dots,i_k\}$ for some $q\in \{1,\dots,k'\}$. Without loss of generality assume that $i_p\notin \{j_1,\dots,j_{k'}\}$ for some $p\in \{1,\dots,k\}$. Therefore, the atom $q_{_{i_p}}\notin x^\perp$ but the atom $q_{_{i_p}}\in y^\perp$. Hence $x^\perp\neq y^\perp$, a contradiction. Therefore $x\approx y$. Thus $x\sim y$ implies that $x\approx y$. 
\end{proof}

\begin{remark} From the above discussion, we conclude that the set $P_{i_1i_2\dots i_k}$ for some $\{i_1,i_2\dots,i_k\}\subseteq \{1,2,\dots,n\}$, is nothing but the equivalence class, say $[a]$, where $a\in P_{i_1i_2\dots i_k}$,  under the equivalence relation $\sim$, and vice-versa.
	The set
	of equivalence classes under $\approx$ of $P\setminus \{0\}$ will be denoted by $$[P^\approx]'=\{~P_{i_1i_2\dots i_k}~ \vert~~ \{i_1,i_2\dots,i_k\}\subseteq \{1,2,\dots,n\}, \text{and } P_{i_1i_2\dots i_k}\neq\emptyset~~\}.$$  Now, we  set $[P^\approx]= [P^\approx]'\cup P_0$. Define a relation $\leq$ on $[P^\approx]$ as follows. $P_{i_1i_2\dots i_k}\leq P_{j_1j_2\dots j_m}$  if and only if $b^\perp\subseteq a^\perp$, for some $a\in P_{i_1i_2\dots i_k}$ and for some $ b\in P_{j_1j_2\dots j_m}$, where $P_{i_1i_2\dots i_k}, P_{j_1j_2\dots j_m}\in [P^\approx]' $  and $P_0\leq P_{i_1i_2\dots i_k}$ for all $\{i_1,i_2\dots,i_k\}\subseteq \{1,2,\dots,n\}$. It is not very difficult to prove that $([P^\approx], \leq)$ is a poset. The least element of  $([P^\approx], \leq)$ is $P_0$ and if $P$ has the greatest element 1, then the greatest element of the poset $([P^\approx],\; \leq)$ is $P_{12\dots n}$ (see Proposition \ref{porder} for details). In view of Lemma \ref{eqsame},    the   posets $([P^\approx], \leq)$ and $([P^\sim],\leq)$ are same. \textbf{Henceforth, without any distinction, we write $[P^\approx]=[P^\sim]$ by simply $[P]$.}

	\vskip 5truept 
	
	Let $P$ is a poset with 0. Since $[P]$ is a poset with the least element $[0]=P_0$, and except $[d]$ (where $d\in \mathcal{D}$), every element of $[P]$ is a zero-divisor. Note that $\mathcal{D}$ may be empty.  Hence the zero-divisor graphs $G([P])$ and $G^*([P])$ of the poset $[P]$ are same, that  is, $G([P])=G^*([P])$. Hence afterwards, we write $G([P])$ for the zero-divisor graph of $[P]$. Clearly, $a$ and $b$ are adjacent in $G(P)$ if and only if $[a]$ and $[b]$ are adjacent in $G([P])$. More about the inter relationship between the properties of $G(P)$ and $G([P])$ are mentioned in Lemma \ref{property}.
\end{remark}	
For this purpose, we need the following definition.

\vskip 5truept 

\begin{definition}[West \cite{west}]
	A set $I$  of vertices of a graph $G$ is said to be \textit{independent} if no two vertices $u$ and $v$ in $I$ are adjacent in $G$. If  $|I|=n$, then we denote $I_n$, the independent set on $n$ vertices. The maximum cardinality of an independent set of vertices of $G$ is called the \textit{vertex-independence number} and is denoted by $\alpha(G)$.
\end{definition}

\begin{proposition}\label{porder}
	Let $P$ be a  poset. Then $P_{i_1i_2\dots i_k}\leq P_{j_1j_2\dots j_m}$ in $[P]$ if and only if  $\{i_1,i_2,\dots, i_k\}\subseteq \{j_1,j_2,\dots, j_m\}$.
\end{proposition}

\begin{proof}
	Let  $P_{i_1i_2\dots i_k}\leq P_{j_1j_2\dots j_m}$ in $[P]$. This gives that $b^\perp\subseteq a^\perp$, for some $a\in P_{i_1i_2\dots i_k}$, and for some $ b\in P_{j_1j_2\dots j_m}$. It is easy to verify that  $q_{_{l_1}},q_{_{l_2}},\dots,q_{_{l_s}}$ are the atoms in $b^\perp$, where $\{l_1,\dots,l_s\}=\{1,\dots,n\}\setminus \{j_1,\dots,j_m\}$. Note that for every $t$, $l_t\notin\{j_1,\dots,j_m\}$, otherwise $q_{_{l_t}}\in b^\ell \cap b^\perp=\{0\},$ a contradiction. Since $q_{_{l_1}},q_{_{l_2}},\dots,q_{_{_{l_s}}}\in b^\perp\subseteq a^\perp$. From this, we have $l_1,\dots,l_s\notin \{i_1,\dots,i_k\}$. Hence, we get $\{i_1,\dots,i_k\}\subseteq \{j_1,\dots,j_m\}$.
	
	Conversely, assume that $\{i_1,\dots,i_k\}\subseteq \{j_1,\dots,j_m\}$. We want to prove that $P_{i_1i_2\dots i_k}\leq P_{j_1j_2\dots j_m}$ in $[P]$. For this, let $a\in P_{i_1i_2\dots i_k}$ and $b\in P_{j_1j_2\dots j_m}$. Then $a^\perp =\biggl(\bigcup\limits_{l\neq i_1,i_2,\dots,i_k}\bigl\{q_{_l}\bigr\}^u\biggr)\mathbin{\Big\backslash} \biggl(\bigcup\limits_{s=1}^k\bigl\{q_{_{i_s}}\bigr\}^u\biggr)$ and $b^\perp =\biggl(\bigcup\limits_{l\neq j_1,j_2,\dots,j_k}\bigl\{q_{_l}\bigr\}^u\biggr)\mathbin{\Big\backslash} \biggl(\bigcup\limits_{s=1}^m\bigl\{q_{_{j_s}}\bigr\}^u\biggr)$. Since $\{i_1,\dots,i_k\}\subseteq \{j_1,\dots,j_m\}$. Then we get, $b^\perp\subseteq a^\perp$. Hence, $P_{i_1i_2\dots i_k}\leq P_{j_1j_2\dots j_m}$ in $[P]$.\end{proof}

\vskip 5truept 

\begin{corollary}\label{adj}
	Let $G([P])$ be the zero-divisor graph of $[P]$. Then	$P_{i_1i_2\dots i_k}$ and $P_{j_1j_2\dots j_m}$ are adjacent in $G([P])$ if and only if $\{i_1,i_2,\dots, i_k\}\cap \{j_1,j_2,\dots, j_m\}=\emptyset$.
\end{corollary}

\begin{proof} Let $P_{i_1i_2\dots i_k}$ and $P_{j_1j_2\dots j_m}$ are adjacent in $G([P])$. Suppose on contrary that $\{i_1,i_2,\dots, i_k\}\cap \{j_1,j_2,\dots, j_m\}\neq \emptyset$, i.e., there exists $l\in \{i_1,i_2,\dots, i_k\}\cap \{j_1,j_2,\dots, j_m\}$. By Proposition \ref{porder},  $P_l\leq P_{i_1i_2\dots i_k}$ and  $P_l\leq P_{j_1j_2\dots j_m}$. This gives that $\{P_0,P_l\}\subseteq\{P_{i_1i_2\dots i_k}, ~ P_{j_1j_2\dots j_m}\}^\ell$, a contradiction, as $P_{i_1i_2\dots i_k}$ and $P_{j_1j_2\dots j_m}$ are adjacent in $G([P])$.  Thus, $\{i_1,i_2,\dots, i_k\}\cap \{j_1,j_2,\dots, j_m\}=\emptyset$.
	
	Conversely, we assume that  $\{i_1,i_2,\dots, i_k\}\cap \{j_1,j_2,\dots, j_m\}=\emptyset$. We want to prove that  $P_{i_1i_2\dots i_k}$ and $P_{j_1j_2\dots j_m}$ are adjacent in $G([P])$, i.e., $\{P_{i_1i_2\dots i_k}, ~ P_{j_1j_2\dots j_m}\}^\ell=\{P_0\}$. Suppose not, then there exists $P_l\not=P_0$ such that $P_l\in\{P_{i_1i_2\dots i_k},~ P_{j_1j_2\dots j_m}\}^\ell$.  By Proposition \ref{porder}, $l\in\{i_1,i_2,\dots, i_k\}\cap \{j_1,j_2,\dots, j_m\}$, a contradiction. Thus, $P_{i_1i_2\dots i_k}$ and $P_{j_1j_2\dots j_m}$ are adjacent in $G([P])$.	\end{proof}

\vskip 5truept 


\begin{lemma}
	Let $q$ be an atom of a  $0$-distributive poset $P$. If $\{q,x_i\}^\ell=\{0\}$ for every $i, 2\leq i\leq n$, then $\{q,\{x_1,\dots,x_n\}^u\}^\ell=\{0\}$. Moreover, there exists $d\in\{x_1,\dots,x_n\}^u$ such that $q\nleq d$.
\end{lemma}

\begin{proof}
	For $n=2$, the result follows from the definition of 0-distributivity.
	\par  Now, we prove for the result for $n=3$. For this, let $\{q,x_1\}^\ell=\{q,x_2\}^\ell=\{q,x_3\}^\ell=\{0\}$. By 0-distributivity, this gives that $\{q,\{x_1,x_2\}^u\}^\ell=\{0\}$. Hence $q\notin \{x_1,x_2\}^{u\ell}$. Therefore, there exist $d_1\in \{x_1,x_2\}^u$ such that $q\nleq d_1$. Thus, $\{q,d_1\}=\{0\}$. Similarly, there exist $d_2\in \{x_1,x_3\}^u$ such that $q\nleq d_2$ and $\{q,d_2\}=\{0\}$. Since $P$ is $0$-distributive, we have $\{q,\{d_1,d_2\}^u\}^\ell=\{0\}$. Therefore, there exist $d_3\in \{d_1,d_2\}^u$ such that $q\nleq d_3$. Hence $\{q,d_3\}^\ell=\{0\}$. Clearly, $x_1,x_2\leq d_1\leq d_3$ and $x_1,x_3\leq d_2\leq d_3$. This together implies that $x_1,x_2,x_3\leq d_3$. Thus, $d_3\in\{x_1,x_2,x_3\}^u$ and $q\nleq d_3$. Therefore we have $q\notin \{x_1,x_2,x_3\}^{u\ell}$ and $\{q,d_3\}^\ell=\{0\}$. This gives that $\{0\}=q^\ell\cap\{x_1,x_2,x_3\}^{u\ell}$. In this case $d=d_3$. Hence moreover part also proved. 
	
	\par Repeating this procedure, we can prove for any finite $n$. 	\end{proof}

\vskip 5truept

\begin{corollary}\label{pseudoatom}
	Let $ q_{_1} , q_{_2} ,\dots, q_{_n} $ be the all  atoms of a $0$-distributive poset $P$.\\ Then $\{ q_{_i}, \{ q_{_1}, \dots, q_{_{i-1}}, \; q_{_{i+1}},\dots, q_{_n}\}^u\}^\ell=\{0\}$ for all $i$. Moreover, $P_{1\dots(i-1)(i+1)\dots n}\neq\emptyset$ for all $i$.
\end{corollary}

\vskip 5truept

\begin{remark} If $P$ is a $0$-distributive poset with $n$ number of atoms, then $P_{i_1i_2\dots i_k}$ may not be a nonempty set for $2\leq k\leq n-2$, where $\{i_1,i_2,\dots, i_k\}\subseteq \{1,\dots,n\}$.	Let $P$ be a $0$-distributive poset with  $4$ atoms (as shown in Figure \ref{exa1}). Then $P_{ij}=\emptyset$ for $i,j \in \{1,2,3,4\}$.

	\begin{figure}[h]
		\begin{center}

			\begin{tikzpicture}[scale =0.7]
				\centering
				
				\draw [fill=black] (-1.5,0) circle (.1);
				\draw [fill=black] (-.5,0) circle (.1);
				\draw [fill=black] (1.5,0) circle (.1);
				\draw [fill=black] (.5,0) circle (.1);
				\draw [fill=black] (-1.5,1.5) circle (.1);
				\draw [fill=black] (-.5,1.5) circle (.1);
				\draw [fill=black] (1.5,1.5) circle (.1);
				\draw [fill=black] (.5,1.5) circle (.1);
				\draw [fill=black] (0,-1.5) circle (.1);
				\draw [fill=black] (0,3) circle (.1);
				
				\draw (0,-1.5)--(-1.5,0)--(-1.5,1.5)--(0,3)--(-0.5,1.5)--(-0.5,1.5)--(-0.5,0)--(0,-1.5)--(0.5,0)--(0.5,1.5)--(0,3)--(1.5,1.5)--(1.5,0)--(0,-1.5);
				\draw (-1.5,0)--(-0.5,1.5)--(1.5,0)--(1.5,1.5)--(.5,0)--(0.5,1.5);
				\draw (-1.5,0)--(0.5,1.5);
				\draw (-.5,0)--(1.5,1.5);
				\draw (0.5,1.5)--(1.5,0);
				\draw (-0.5,0)--(-1.5,1.5)--(0.5,0);
				
				\draw node [above] at (-1.5,1.5) {$q_{4}^*$};
				\draw node [above] at (-0.6,1.5) {$q_{3}^*$};
				\draw node [above] at (0.6,1.5) {$q_{2}^*$};
				\draw node [above] at (1.5,1.5) {$q_{1}^*$};
				\draw node [below] at (-0.6,0) {$q_{2}$};
				\draw node [below] at (-1.5,-0.05) {$q_{1}$};
				\draw node [below] at (0.65,0) {$q_{3}$};
				\draw node [below] at (1.5,-0.05) {$q_{4}$};
				\draw node [below] at (0,-1.7) {$0$};
				\draw node [above] at (0,3) {$1$};

				
			\end{tikzpicture}
		\end{center}
		\caption{0-distributive poset $P$}\label{exa1}
	\end{figure}
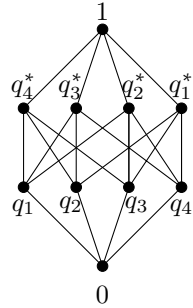
\end{remark}

\begin{proposition}\label{atompseudo}
	If $P$ is a  $0$-distributive  poset, then $P_1,\dots,P_n$ are the atoms of $[P]$ and $ P_{1\dots(i-1)(i+1)\dots n}$ is the pseudocomplement of $P_i$ in $[P]$ for all $i, 1\leq i\leq n$.
\end{proposition}

\begin{proof}
	By Proposition \ref{porder}, $P_1,\dots,P_n$ are the atoms of $[P]$. We prove that $ P_{1\dots(i-1)(i+1)\dots n}$ is the pseudocomplement of $P_i$ in $[P]$, for all $i$. Clearly, $\Bigl\{P_i, \; P_{1\dots(i-1)(i+1)\dots n}\Bigr\}^\ell =\{P_0\}$. Assume that $P_{j_1j_2\dots j_m}\in [P]$ such that $\Bigl\{P_i,P_{j_1j_2\dots j_m}\Bigr\}^\ell=\{P_0\}$. By Corollary \ref{adj}, $i\notin \{j_1,j_2,\dots,j_m\}$. Since $\{i\}\cup\{1,\dots,i-1,i+1,\dots,n\}=\{1,\dots,n\}$ and $\{j_1,j_2,\dots,j_m\}\subseteq \{1,\dots,n\}$. Then $\{j_1,j_2,\dots,j_m\}\subseteq \{1,\dots,i-1,i+1,\dots,n\}$. By Proposition \ref{porder}, we have $P_{j_1j_2\dots j_m}\leq  P_{1\dots(i-1)(i+1)\dots n}$. This proves that $ P_{1\dots(i-1)(i+1)\dots n}$ is pesudocomplement of $P_i$ in $[P]$, for all $i$.\end{proof}

\vskip 5pt 

The following statements 1-4 are essentially proved in \cite{djl} (see Lemma 4.5, Lemma 4.2), statements 5-7 are essentially proved in  \cite{nkvj} (see Lemma 2.1 (6), (8), respectively). We write these statements in terms of $P_{i_1i_2\dots i_k}$. These properties will be used frequently in the sequel.

\begin{lemma}\label{property} 
	The following statements are true.
	\begin{enumerate}
		\item If $q_{_1},q_{_2},\dots,q_{_n}$ are distinct atoms of ${ P}$, then $[ {q}_{_1} ],\dots,[ {q}_{_n} ]$ are distinct atoms of $[{ P}]$. Note that $[q_{_i}]=P_{i}$, for every $i\in\{1,\dots,n\}$ .
		\vskip 5truept 
		\item If $ {a}\leq {b}$ in ${ P}$, then $[{a} ]\leq [{b} ]$ in $[{ P}]$. Moreover, $P_{i_1i_2\dots i_k}\leq P_{j_1j_2\dots j_m}$ in $[P]$ if and only if  $\{i_1,i_2,\dots, i_k\}\subseteq \{j_1,j_2,\dots, j_m\}$.
		
		\vskip 5truept 
		\vskip 5truept

		\item $\{{a,b} \}^\ell=\{{0} \}$ in ${ P}$ if and only if $\{[{a} ],[{b} ]\}^\ell=\{[{0} ]\}$ in $[{ P}]$. Note that the lower cones are taken in the respective posets. $P_{i_1i_2\dots i_k}$ and $P_{j_1j_2\dots j_m}$ are adjacent in $G([P])$ if and only if $\{i_1,i_2,\dots, i_k\}\cap \{j_1,j_2,\dots, j_m\}=\emptyset$. Further, $ {a} \in V(G({ P}))$ if and only if $[{a} ] \in V(G([{ P}]))$.

		\vskip 5truept  
		
		\item Let $[{a} ]\in V(G([{ P}]))$. Then for any ${x,y}\in [{a} ]$, $\{{x,y} \}^\ell\neq \{{0} \}$ in ${ P}$. Hence vertices of $[{a} ]$ forms an independent set in $G({ P})$. Further, if  $\{[{a} ],[{b} ]\}^\ell=\{[{0} ]\}$ in $[{ P}]$, then for any ${x}\in [{a} ]$ and for any  $ {y}\in [{b} ]$,    $\{{x,y} \}^\ell=\{{0} \}$ in ${ P}$. In particular, $[{a} ]$ and $ [{b}]$ are adjacent in $G({ [P]})$ with $|[{a} ]|=m$, $|[{b} ]|=n$, then the vertices of $[{a} ]$ and $[{b} ]$ forms an induced complete bipartite subgraph $K_{m,n}$ of $G({ P})$. Moreover, for any $x, y \in [a]$, deg$_{G(P)}(x)= $ deg$_{G(P)}(y)$.

		\vskip 5truept 
		
		\item If ${q}_{_1},\dots, {q}_{_n}$ are the atoms of ${ P}$, then $A_{{ q}_{_i}}$ is set of an independent vertices of $G({P})$ for every $i\in\{1,2,\dots,n\}$, and $V(G({P}))=\bigcupdot   A_{{q}_{_{i}} }$. Also, $q_{_i}^u\setminus \mathcal{D}$ is set of an independent vertices of $G(P)$ for every $i\in\{1,\dots,n\}$.
		
		\vskip 5truept

		\item The induced subgraph of $G([P])$ on the set $\{P_1,P_2,\dots,P_n\}$ is a complete graph on $n$ vertices and the induced subgraph of $G(P)$ on the set $\bigcup\limits_{i=1}^{n}P_i$ is a complete $n$-partite graph. Therefore the induced subgraph of $G([P])$ on the set $\{P_1,\dots,P_n,P_{23\dots n},\dots, P_{12\dots (n-1)}\}$ is a complete graph on $n$ vertex $K_n$ with exactly one pendent vertex attached to each vertex of $K_n$ (See Figure \ref{fig1} in the case $n=3$), provided $P_{23\dots n},\dots, P_{12\dots (n-1)}$ are nonempty sets.

		\item Let $P$ be a $0$-distributive poset. Then $P_i$ be the atom of $[P]$ and	$P_{1\dots(i-1)(i+1)\dots n}$ is the pseudocomplement of $P_i$ in $[P]$ for all $i$. 	
		Moreover in $G([P])$, the vertex $P_{1\dots(i-1)(i+1)\dots n}$ is adjacent to only $P_i$,   that is, $P_{1\dots(i-1)(i+1)\dots n}$ is the only pendent vertex to $P_i$ in $G([P])$, for all $i$.
		Therefore the set $\{P_{23\dots n}, P_{13\dots n}, P_{12\dots(n-1)}\}$ of $G([P])$ is an independent set  of vertex, that is,   the induced subgraph of $G^c([P])$ on the set $\{P_{23\dots n}, P_{13\dots n}, P_{12\dots(n-1)}\}$ is a complete graph on $n$ vertex. 
	\end{enumerate}
\end{lemma}

\noindent\textbf{Notation:}
Let $q_{_1},q_{_2},\dots,q_{_n}$ be all  atoms of a  $0$-distributive  poset $P$. Then $[q_1]=P_1,\dots,[q_n]=P_n$ are the atoms of $[P]$ and $[q_i]^*= P_{1\dots(i-1)(i+1)\dots n}$ is the pseudocomplement of $[q_i]=P_i$ in $[P]$ for all $i$. Moreover, one can check that if $[P]$ has the greatest element, then $[q_i]^*= P_{1\dots(i-1)(i+1)\dots n}$ is the complement of $[q_i]=P_i$ in $[P]$ for all $i$.

\vskip 5truept 

The following result is useful to prove that a finite lattice $L$ is pseudocomplemented. One needs to look at whether all the atoms have pseudocomplements.

\begin{lemma}[{Chameni-Nembua and Monjardet \cite[Lemma 3]{cm}}]\label{mon} 
	A finite lattice $L$ is pseudocomplemented if and only if each atom of $L$ has the pseudocomplement.
\end{lemma}

\begin{lemma}[{Joshi and Mundlik \cite[Lemma 2.5]{jm}}]\label{ssc} The poset $[P]$ of all equivalence classes of a poset $P$ with 0 is an SSC poset.
\end{lemma}

\begin{theorem}[Janowitz \cite{jan}]\label{boolean}
	Every pseudocomplemented, SSC lattice is Boolean.
\end{theorem}

\begin{theorem}\label{nboolean}
	Let $L$ be a $0$-distributive bounded lattice with finitely many atoms. Then $[L]$ is a Boolean lattice.
\end{theorem}

\begin{proof}
	Let $L$ be a $0$-distributive bounded lattice with finitely many atoms $q_{_1},q_{_2},\dots,q_{_n}$. Then by Lemma \ref{ssc}, $[L]$ is SSC with $P_1,P_2,\dots,P_n$ as atoms of $[L]$. Clearly, $[L]$ is a atomic lattice. This together with $[L]$ is  SSC, we observe that $[L]$ is an atomistic lattice with $n$ atoms. Since every element of an atomistic lattice is a join of atoms below it, so $[L]$ may have at most $2^n$ elements. Thus $[L]$ is a finite atomistic lattice. 
	By Corollary \ref{pseudoatom}, the sets $P_{23\dots n},P_{13\dots n}, \dots, P_{12\dots(n-1)}$ are nonempty. Hence these sets are the elements of $[L]$.  By Lemma \ref{property}(7),  $P_{1\dots (i-1)(i+1)\dots n}$ is the pseudocomplement of the atom $P_i$ for all $i$.  By Lemma \ref{mon},  $[L]$ is pseudocomplemented. Hence, by Theorem  \ref{boolean},  $[L]$ is a Boolean lattice.	\end{proof}
\vskip 5truept 

\begin{remark}
	Note that the above result fails if we remove the condition of finiteness of atoms. For this, consider the set $L=\{X \subseteq \mathbb{N} ~|~ |X| < \infty\} \cup \{\mathbb{N}\}$. Then $L$ is a bounded 0-distributive lattice under set inclusion as a  partial order. Since $L$ is a 0-distributive, atomistic lattice,    $[L] \cong L$ (cf. \cite[Theorem 2.4]{jm}) is a lattice. However, $[L]$ is not Boolean. 
\end{remark}

The following remark gives the relation between $G(P)$ and $G^*(P)$. For this purpose, we need the following definitions.

\begin{definition}[\cite{west}]
	The \textit{join} of two graphs $G$ and $H$ is a graph formed from disjoint copies of $G$ and $H$ by connecting each vertex of $G$ to each vertex of $H$. We denote the join of graphs 	$G$ and $H$ by $G\vee H$.
	The \textit{disjoint union} of graphs is an operation that combines two or more graphs to form a larger graph. It is analogous to the disjoint union of sets and is constructed by making the vertex set of the result be the disjoint union of the vertex sets of the given graphs and by making the edge set of the result be the disjoint union of the edge sets of the given graphs. Any disjoint union of two or more nonempty graphs is necessarily disconnected. The disjoint union is also called the graph sum. If $G$ and $H$ are two graphs, then $G+H$  denotes their disjoint union.	
\end{definition}
\begin{remark} \label{zdgczdg}
	Let $P$ be a poset. Then $G^*(P)=G(P)+I_m$ and $G^{*c}(P)=G^c(P)\vee K_m$, where $m=|\mathcal{D}\setminus \{1\}|$, if $P$ has the largest element 1; otherwise $m=|\mathcal{D}|$. If $P$ has $n$ atoms, then $|\mathcal{D}|=|P_{12\dots n}|$. 
\end{remark}

\vskip 5truept

\section{Chordal zero-divisor graphs }
A \textit{chord} of a cycle $C$ of graph $G$ is an edge that is not in $C$ but has both its end vertices in $C$. A graph $G$ is \textit{chordal} if every cycle of length at least $4$ has a chord, \textit{i.e.},  $G$ is chordal if and only if it does not contains induced cycle of length at least $4$. We assume that a null graph (without edges and vertices) is chordal.

Let $G$ be a finite graph. The set $\{u\in V(G)~~|~~u-v\in E(G)\}$ be the neighborhood of a vertex $v$ in graph $G$, denoted by $N_G(v)$. If there is no ambiguity about the graph under consideration, we write $N(v)$.   

Define a relation  on $G$ such that $u\simeq v$ if and only if either $u=v$, or  $u-v\in E(G)$ and $N(u)\setminus \{v\}= N(v)\setminus \{u\}$. Clearly, $\simeq$ is an equivalence relation on $V(G)$. The equivalence class of $v$ is the set $\{u\in V(G)~~|~~u\simeq v\}$, denoted by $[v]^\simeq$. Denote the set $\{[v]^\simeq~~|~~ v\in V(G)\}$ by $G_{red}$. Define $[u]-[v]$ is an edge in $E(G_{red})$ if and only if $u-v\in E(G)$, where $[u]\neq [v]$.

	D. F. Anderson and John LaGrange \cite{adlag} studied the few equivalance relation on graphs, rings, and in particular, on   zero-divisor graphs of rings. One of them is on a finite simple graph $G$ which is as follows:  Define a relation  on $G$ such that $u~\Theta~ v$ if and only if  $N(u)\setminus \{v\}= N(v)\setminus \{u\}$. Note that the relations $\simeq$ and $\Theta$ on a simple finite graph are different.
	
\begin{remark}\label{gred}
	It is easy to observe that, if $G^c(P)$ be  the complement of the zero-divisor graph $G(P)$, then $(G^c(P))_{red}=G^c([P])$. 
	
\end{remark}	
\begin{theorem} \label{gchord}
	Let $G$ be a finite graph. Then $G$ is chordal if and only if $G_{red}$ is chordal.
\end{theorem}

\begin{proof} Let $G$ be a chordal graph. Suppose on contrary that $G_{red}$ is not chordal. Hence  $G_{red}$ contains an induced  cycle of length at least $4$. Let the set vertices $[a_1]-[a_2]-\dots-[a_n]-[a_1]$, where $n\geq 4$, form an induced cycle of length $n$ in $G_{red}$. By the definition of $G_{red}$, we have the set of vertices $a_1-a_2-\dots-a_n-a_1$ forms an induced cycle of length $n \geq 4$ in $G$, a contradiction. Thus $G_{red}$ is chordal.

	\par Conversely,  suppose that $G_{red}$ is a chordal graph.  Suppose on contrary that $G$ is not chordal. Hence, $G$  contains an induced  cycle of length at least $4$.  Let $a_1-a_2-\dots-a_n-a_1$ be the smallest  induced  cycle of length $n\geq 4$ in $G$. This gives that $\{a_{(i-1)mod ~n},a_{(i+1)mod~n}\}\subseteq N(a_i)$. We first prove that $[a_i]\neq[a_j]$ for $i\neq j$. Clearly, $a_i\not= a_j$. If $[a_i]=[a_j]$ for some $i\neq j$, then by the definition of equivalence relation $a_i-a_j$ and $N(a_i)\setminus \{a_j\} =N(a_j)\setminus \{a_i\}$. If $a_j \notin\{a_{i-1}, a_{i+1}\}$, then this implies that $\{a_{(i-1)mod ~n},a_{(i+1)mod~n}\}\subseteq N(a_j)$, that is, $a_j$ is adjacent to $a_i, a_{(i-1)mod ~n}$ and $a_{(i+1)mod~n}$, a contradiction to the minimality of the length of the cycle. Thus, in this case,  $[a_i]\neq[a_j]$ for $i\neq j$.
	
	Now, assume that $a_j \in\{a_{i-1}, a_{i+1}\}$. Without loss of generality,  $a_j=a_{i+1}$. Since $\{a_{i-1},a_{j}\}\subseteq N(a_i)$ and $N(a_i)\setminus \{a_j\} =N(a_j)\setminus \{a_i\}$, we have $a_{i-1}\in N(a_j)$, that is, $a_j$ is adjacent to $a_i$, and $a_{i-1}$, a contradiction to the minimality of the length of the cycle. Thus, in this case,  $[a_i]\neq[a_j]$ for $i\neq j$. 
	
	Therefore, in both the cases,  $[a_i]\neq[a_j]$ for $i\neq j$.	 By the definition of $G_{red}$, we have an induced cycle  $[a_1]-[a_2]-\dots-[a_n]-[a_1]$ of length $n\geq 4$ in $G_{red}$, a contradiction. Thus, $G$ is a chordal graph.		 	\end{proof}

In view of Remark \ref{gred} and Theorem \ref{gchord}, we have the following corollary.

\begin{corollary}
	Let $P$ be a finite poset. Then $G^c(P)$ is a chordal graph if and only if $G^c([P])$ is a chordal graph.
\end{corollary}

\begin{remark}\label{obs1} It is easy to observe that $G+ I_m$ is a chordal graph if and only if $G$ is chordal if and only if $G\vee K_m$ is chordal. In particular, $G(P)$  is chordal if and only if $G^*(P)$ is chordal. Also,  $G^c(P)$  is chordal if and only if $G^{*c}(P)$ is a chordal graph.   
	
\end{remark} 

It should be noted that if a poset $P$ has exactly one atom, then $G^*(P)$ is an empty graph (without edges) of size $|P\setminus \{0,1\}|$, if $P$ has the greatest element, otherwise 
of size $|P\setminus \{0\}|$. However, if $P$ has exactly one atom,  the zero-divisor graph $G(P)$ is a null graph (without vertices and edges) that we can assume to be a chordal graph.
Thus, with this preparation, we are ready to prove statements $\textbf{(A)}$ and $\textbf{(B)}$ of our first main result.

\begin{proof}[\underline{\underline{Proof of Theorem \ref{zdgchordal}}}] \textbf{(A)} 
	Since $P$ be a finite poset such that  $[P]$ is Boolean, we have $P_{i_1i_2\dots i_k}\neq \emptyset$ for any nonempty set $\{i_1,i_2,\dots, i_k\}\subsetneqq \{1,2,\dots,n\}$.
	
	Suppose that $G(P)$ is chordal. Then we prove that the number $n$ of atoms of $P$ is $\leq 3$. We assume that $n\geq 4$. Then using the statements (3) and (5) of Lemma \ref{property}, we show that there exists an induced cycle of length $4$, where $q_1\in P_1,\   q_2\in P_2, \  x_{14}\in P_{14}, \  x_{23}\in P_{23}$ shown in Figure \ref{chordal1}(A). Thus $n\leq 3$. 
	
	We discuss the following three cases.
	
	If $n=1$, then $G(P)$ is a null graph, and hence $G(P)$ is chordal, as per the assumption.
	
	Now, assume that $n=2$. Then $G(P)$ is a complete bipartite graph $K_{m_1,m_2}$, where $|P_i|=m_i$ for $i\in \{1,2\}$. We show that one of $m_1$ and $m_2$ is $1$. If not, then $m_i\geq 2$ for all $i\in \{1,2\}$. Let $x_{11},x_{12}\in P_1$ and let $x_{21},x_{22}\in P_2$. Figure \ref{chordal1}(B) shows an induced cycle $x_{11} - x_{21} - x_{12} - x_{22} - x_{11}$ of length $4$, a contradiction. This proves that  one of $m_1$ and $m_2$ is $1$.
	
	Let $n=3$. We have to show that $|P_i|=1$ for all $i\in \{1,2,3\}$. If not, then  $|P_i|\geq 2$ for some $i\in \{1,2,3\}$. Without lose of generality, we assume that $|P_1|\geq 2$. Let $x_{11}, x_{12}\in P_1$, $x_{23}\in P_{23}$, and $q_2\in P_2$. Figure \ref{chordal1}(C) shows an induced cycle  $x_{11} - x_{23} - x_{12} - q_2 - x_{11}$ of length $4$, a contradiction. This proves that $|P_i|=1$ for all $i\in \{1,2,3\}$.

	\begin{figure}[h]
		\begin{center}
			
			\begin{tikzpicture}[scale =.8]

				\draw [fill=black] (-1,0) circle (.1);
				\draw [fill=black] (1,0) circle (.1);
				\draw [fill=black] (-1,2) circle (.1);
				\draw [fill=black] (1,2) circle (.1);
				
				\draw (-1,0) --(1,0)--(1,2)--(-1,2)--(-1,0);
				
				\draw node [below] at (-1,0) {$q_1$};
				\draw node [below] at (1,0) {$q_2$};
				\draw node [above] at (1,2) {$x_{14}$};
				\draw node [above] at (-1,2) {$x_{23
					}$};
				
				\draw node [below] at (0,-1) {$(A)$};

				\begin{scope}[shift={(4,0)}]

					\draw [fill=black] (-1,0) circle (.1);
					\draw [fill=black] (1,0) circle (.1);
					\draw [fill=black] (-1,2) circle (.1);
					\draw [fill=black] (1,2) circle (.1);
					
					\draw (-1,0) --(1,0)--(1,2)--(-1,2)--(-1,0);
					
					\draw node [below] at (-1,0) {$x_{11}$};
					\draw node [below] at (1,0) {$x_{21}$};
					\draw node [above] at (1,2) {$x_{12}$};
					\draw node [above] at (-1,2) {$x_{22
						}$};
					
					\draw node [below] at (0,-1) {$(B)$};

				\end{scope}

				\begin{scope}[shift={(8,0)}]

					\draw [fill=black] (-1,0) circle (.1);
					\draw [fill=black] (1,0) circle (.1);
					\draw [fill=black] (-1,2) circle (.1);
					\draw [fill=black] (1,2) circle (.1);
					
					\draw (-1,0) --(1,0)--(1,2)--(-1,2)--(-1,0);
					
					\draw node [below] at (-1,0) {$x_{11}$};
					\draw node [below] at (1,0) {$x_{23}$};
					\draw node [above] at (1,2) {$x_{12}$};
					\draw node [above] at (-1,2) {$q_2$};
					
					\draw node [below] at (0,-1) {$(C)$};

				\end{scope}

				\begin{scope}[shift={(13,0)}]
					
					\draw [fill=black] (-1,0) circle (.1); \draw [fill=black] (1,0) circle (.1);\draw [fill=black] (0,2) circle (.1); \draw [fill=black] (2,.5) circle (.1);  \draw [fill=black] (-2,.5) circle (.1);  \draw [fill=black] (-0.5,3) circle (.1);
					\draw [fill=black] (2,-.5) circle (.1);  \draw [fill=black] (-2,-.5) circle (.1);  \draw [fill=black] (0.5,3) circle (.1);
					
					\draw [fill=black] (0.2,3) circle (.025);
					\draw [fill=black] (-0.2,3) circle (.025);
					\draw [fill=black] (0,3) circle (.025);
					
					\draw [fill=black] (-2,0.2) circle (.025);
					\draw [fill=black] (-2,-0.2) circle (.025);
					\draw [fill=black] (-2,0) circle (.025);

					\draw [fill=black] (2,0.2) circle (.025);
					\draw [fill=black] (2,-0.2) circle (.025);
					\draw [fill=black] (2,0) circle (.025);

					\draw (-1,0) --(1,0)--(0,2)--(-1,0)--(-2,0.5);
					\draw	(-1,0)--(-2,-0.5);
					\draw (2,-0.5)--	(1,0)--(2,0.5);	\draw (0.5,3)--(0,2)--(-0.5,3);

					\draw (-2,0) ellipse (.2 and 1);
					\draw (2,0) ellipse (.2 and 1);
					\draw (0,3) ellipse (1 and 0.2);

					\draw node [below] at (-1,-.1) {$q_1$}; 
					\draw node [below] at (1,-.1) {$q_2$};
					
					\draw node [right] at (0,2) {$q_3$};

					\draw node [below] at (0,-1) { $(D)$};
					
					\draw node [left] at (-2.2,0) {$P_{23}$};
					
					\draw node [right] at (2.2,0) {$P_{13}$};
					
					\draw node [above] at (0,3.2) {$P_{12}$}; 
					
				\end{scope}

			\end{tikzpicture}
		\end{center}	
		\caption{}\label{chordal1}
	\end{figure}

	Conversely, suppose that one of the condition $(1), (2), (3)$ is satisfied. One can see that Condition $(1)$ implies that $G(P)$ is a null graph, and thus $G(P)$ is chordal.  If Condition $(2)$ holds, then $G(P)$ is a complete bipartite graph $K_{m_1, m_2}$, where $m_1=1$ or $m_2=1$. This implies that $G(P)$ is chordal. If Condition $(3)$ satisfied, then the vertex set of $G([P])$ is the set $\{P_1, P_2, P_3, P_{12}, P_{13}, P_{23}\}$ and $|P_i|=1$ for all $i\in \{1,2,3\}$. Then the  zero-divisor graph $G(P)$ is shown in  Figure \ref{chordal1}(D). Further, the vertices of $P_{ij}$ forms an independent set for $i, j \in \{1,2,3\}$ and $i <j$.  Clearly, in this case also $G(P)$ is chordal.

	\textbf{(B)} 	Suppose $G^c(P)$ is a chordal graph. We have to prove that $P$ has at most 3 atoms. Suppose on contrary the number $n$ of atoms is $\geq 4$. Choose $x_{12}\in P_{12}$, $x_{14}\in P_{14}$, $x_{34}\in P_{34}$, $x_{23}\in P_{23}$. Then we can have an induced cycle $x_{12} - x_{14} - x_{34} - x_{23} - x_{12}$  of length $4$ as shown in Figure \ref{chordal2}(A), a contradiction to the fact that $G^c(P)$ is chordal. Thus $n\leq 3$. 
	
	Conversely, suppose that the number $n$ of atoms in $P$ is at most  $3$. We must prove that $G^c(P)$ is a chordal graph. 
	
	If $n=1$, then $G^c(P)$ is a null graph. Hence $G^c(P)$ is a chordal graph.		
	
	Now, assume that  $n=2$. Then $G^c(P)$ is a union of two complete  graphs $K_{m_1}+K_{m_2}$, where $|P_i|=m_i$ for $i\in \{1,2\}$. This implies that $G^c(P)$ is chordal graph. 	
	
	Lastly, assume that $n=3$. By Remark \ref{gred} and Theorem \ref{gchord}, $G^c(P)$ is chordal if and only if $G^{c}([P])$ is chordal. From Figure \ref{chordal2}(B), it is easy to observe that $G^{c}([P])$ is chordal. Therefore $G^c(P)$ is chordal.

	\begin{figure}[h]
		\begin{center}
			
			\begin{tikzpicture}[scale =.71]

				\draw [fill=black] (-1,0) circle (.1);
				\draw [fill=black] (1,0) circle (.1);
				\draw [fill=black] (-1,2) circle (.1);
				\draw [fill=black] (1,2) circle (.1);
				
				\draw (-1,0) --(1,0)--(1,2)--(-1,2)--(-1,0);
				
				\draw node [below] at (-1,0) {$x_{12}$};
				\draw node [below] at (1,0) {$x_{14}$};
				\draw node [above] at (1,2) {$x_{34}$};
				\draw node [above] at (-1,2) {$x_{23
					}$};
				
				\draw node [below] at (0,-1.7) {$(A)$};

				\begin{scope}[shift={(5,0)}]
					
					\draw [fill=black] (-1,0) circle (.1); \draw [fill=black] (1,0) circle (.1);\draw [fill=black] (0,2) circle (.1); \draw [fill=black] (1.3,1.5) circle (.1);  \draw [fill=black] (-1.3,1.5) circle (.1);  \draw [fill=black] (0,-1) circle (.1);
					
					\draw (-1,0) --(1,0)--(0,-1)--(-1,0)-- (-1.3,1.5)--(0,2)--(-1,0)--(1,0)--(1.3,1.5)--(0,2)--(1,0);
					
					\draw node [left] at (-1,0) {$P_{13}$}; 
					\draw node [right] at (1,0) {$P_{12}$};
					\draw node [above] at (-1.3,1.5) {$P_{3}$};
					\draw node [above] at (1.3,1.5) {$P_{2}$};
					\draw node [above] at (0,2) {$P_{23}$};
					\draw node [below] at (0,-1.1) {$P_{1}$};
					
					\draw node [below] at (0,-1.7) {(B) ~~The graph $G^{c}([P])$};
					
				\end{scope}
				
			\end{tikzpicture}
		\end{center}
		\caption{}\label{chordal2}
	\end{figure}
\end{proof}

In the following remark, we provide the class of posets $P$ such that $[P]$ is a Boolean lattice.		
\begin{remark}\label{bolrem}
	We observe that a finite 0-distributive lattice $L$ is pseudocomplemented. Hence in view of Theorem \ref{nboolean}, $[L]$ is Boolean. Another class of posets $P$ for which $[P]$ is a Boolean lattice is  $\textbf{P}=\prod\limits_{i=1}^nP^i$, where $P^i$  be a finite bounded poset  with $Z(P^i)=\{0\}$ for every $i$. This follows from Lemma \ref{product}. 
\end{remark}	

\begin{lemma} [{Khandekar and Joshi \cite[Lemma 2.1(3)]{nkvj}}] \label{product}
	Let $\textbf{P}=\prod\limits_{i=1}^nP^i$, where $P^i$  be a finite bounded poset  with $Z(P^i)=\{0\}$ for every $i$. Then $[\textbf{P}]$ is a Boolean lattice and $|P_{i_1\dots i_k}|= \prod\limits_{i=i_1}^{i_k}\big(|P^i|-1\big)$, where $\{i_1, \dots, i_k\}\subseteq \{1,\dots, n\}$.
\end{lemma}

In view of Theorem \ref{zdgchordal}, Theorem \ref{nboolean},  and Lemma \ref{product}, we have the following corollaries.

\begin{corollary} \label{zdgchordlattice}
	Let $L$ be a finite $0$-distributive  lattice. Then
	
	\textbf{(A)} $G(L)$ is chordal if and only if one of the following hold:
	
	\begin{enumerate}
		\item $L$ has exactly one atom;
		
		\item $L$ has exactly two atoms with $|L_i|=1$ for some $i\in \{1,2\}$ (see $(\circledcirc)$);

		\item $L$ has exactly three atoms with $|L_i|=1$ for all $i\in \{1,2,3\}$ (see $(\circledcirc)$).
	\end{enumerate}
	
	\textbf{(B)} $G^{c}(L)$ is chordal if and only if number of atoms of $L$ are at most $3$.	.	
	
\end{corollary}

\begin{corollary} \label{zdgchordproduct}
	Let $\textbf{P}=\prod\limits_{i=1}^nP^i$, where $P^i$  be a finite bounded poset with $Z(P^i)=\{0\}$ for every $i$. Then
	
	\textbf{(A)} $G(\mathbf{P})$ is chordal if and only if one of the following holds:
	
	\begin{enumerate}
		\item $n=1$;
		
		\item  $n=2$ with $|P^i|=2$ for some $i\in \{1,2\}$, i.e., $P^i=C_2$ for some $i\in \{1,2\}$;

		\item  $n=3$ with $|P^i|=2$ for all $i\in \{1,2,3\}$, i.e., $\textbf{P}=C_2\times C_2\times C_2$.
	\end{enumerate}
	
	\textbf{(B)} $G^c(\mathbf{P})$ is chordal if and only if number of atoms of $\mathbf{P}$ are at most $3$.	.	
	
\end{corollary}

\vskip 25truept

\section{ Perfect Zero-divisor Graphs }

In this section, we prove the characterizations of perfect zero-divisor graphs of ordered sets. A key result to prove the perfectness of zero-divisor graphs of ordered sets is the Strong Perfect Graph Theorem due to Chudnosky et al. \cite{strongperfect}.

\begin{theorem} [Strong Perfect Graph Theorem \cite{strongperfect}]\label{strongperfect}
	A graph $G$  is perfect if and only if neither $G$ nor $G^c$ contains an induced odd cycle of length at least $5$.
\end{theorem}

In view of Theorem \ref{strongperfect}, we have the following corollary. The statement $(1)$ is nothing but the Perfect Graph Theorem due to Lov\'asz \cite{perfectlovasz}.

\begin{corollary} \label{corsperfect}
	Let $G$ be a graph. Then the following statements hold:
	\begin{enumerate}
		\item $G$ is a perfect graph if and only if $G^c$ is a perfect graph.
		
		\item If $G$ is a complete bipartite graph, then $G$ is a perfect graph.
	\end{enumerate}
\end{corollary}

The following result gives the relation between chordal graphs and perfect graphs.

\begin{theorem} [Dirac \cite{dirac}] \label{choedalperfect}
	Every chordal graph is perfect.
\end{theorem}

\begin{theorem} \label{redperfect}
	Let $G$ be a finite graph. Then $G$ is perfect if and only if $G_{red}$ is perfect.
\end{theorem}

\begin{proof} Let $G$ be a perfect graph. By  Strong Perfect Graph Theorem, neither $G$ nor $G^c$ contains an induced odd cycle of length at least $5$. Suppose, on the contrary, that $G_{red}$ is not perfect. By  Strong Perfect Graph Theorem, either $G_{red}$ or $G^c_{red}$ contains an induced odd cycle of length at least $5$. Without loss of generality, we assume that $G_{red}$ contains an induced odd cycle of length at least $5$. 
	
	Let $[a_1]-[a_2]-\dots-[a_n]-[a_1]$  be an induced cycle of odd length $n$, where $n\geq 5$  in $G_{red}$. By the definition of the equivalence relation $\simeq$ and $G_{red}$, we get $a_1-a_2-\dots-a_n-a_1$ an induced odd cycle of length $n\geq 5$ in $G$, a contradiction. Thus $G_{red}$ is perfect.

	\par Conversely,  assume that $G_{red}$ is a perfect graph. By Strong Perfect Graph Theorem, neither $G_{red}$ nor $G^c_{red}$ contains an induced odd cycle of length at least $5$. Suppose on contrary that $G$ is not perfect. Then $G$ or $G^c$ contains an induced odd cycle of length at least $5$. Assume that $G$  contains an induced odd cycle $a_1-a_2-\dots-a_n-a_1$ of length $n$, where $n\geq 5$.  In view of proof of Theorem \ref{cgchord}, we have $[a_i]\neq[a_j]$ for $i\neq j$. By the definition of the equivalence relation $\simeq$ and $G_{red}$, we have an induced odd cycle $[a_1]-[a_2]-\dots-[a_n]-[a_1]$  of length $n \geq 5$ in $G_{red}$, a contradiction.
	

	Assume that $G^c$  contains an induced odd cycle of  length at least $5$. Let $a_1-a_2-\dots-a_n-a_1$forms an induced odd cycle of length $n \geq 5$ in $G^c$. This gives that $\{a_{(i-1)},a_{(i+1)}\}\nsubseteq N_G(a_i)$. 	We first prove that $[a_i]\neq[a_j]$ for $i\neq j$ and $i,  j \in \{1,2, \cdots, n\}$. Clearly, $a_i\not= a_j$. If $[a_i]=[a_j]$ for some $i\neq j$, and $i,\ j  \in \{1,2, \cdots, n\}$, then by the definition of the equivalence relation $\simeq$,  $a_i-a_j$ in $G$ and $N_G(a_i)\setminus \{a_j\} =N_G(a_j)\setminus \{a_i\}$. If $a_j \notin\{a_{i-1}, a_{i+1}\}$, then this implies that $\{a_{(i-1)},a_{(i+1)}\}\nsubseteq N_G(a_j)$, that is, $a_j$ is adjacent to $ a_{(i-1)}$ and $a_{(i+1)}$ in $G^c$. This contradicts the fact that $a_1-a_2-\dots-a_n-a_1$ forms an induced  odd cycle in $G^c$. Thus, in this case,  $[a_i]\neq[a_j]$ for $i\neq j$.
	
	Now, assume that $a_j \in\{a_{i-1}, a_{i+1}\}$. Without loss of generality,  $a_j=a_{i+1}$. This gives that $a_i$ and $a_j$ are not adjacent in $G$.  Thus, in this case,  $[a_i]\neq[a_j]$ for $i\neq j$. Therefore $[a_i]\neq[a_j]$ for $i\neq j$. 	 By the definitions of the equivalence relation $\simeq$ and $G_{red}$, we have an induced odd cycle of length $n\geq 5$ $[a_1]-[a_2]-\dots,[a_n]-[a_1]$ in $G^c_{red}$, a contradiction. 
	
	Thus in both cases, we get a contradiction. Therefore neither $G$ nor $G^c$ contains an induced odd cycle of length at least $5$. Thus, $G$ is a perfect graph.	 	\end{proof}

Bagheri et al. \cite{bagheri} considered the following relation on a graph $G$: $u\approxeq v$ if and only if $N_G(u)= N_G(v)$. Clearly, $\approxeq$ is an equivalence relation on $V(G)$. The equivalence class of $v$ is the set $\{u\in V(G)~~|~~u\approx v\}$, denoted by $[v^\approxeq]$. Denote the set $\{[v^\approxeq]~~|~~ v\in V(G)\}$ by $[V(G)]$. Define $[u^\approxeq]-[v^\approxeq]$ is an edge in $E([G])$ if and only if $u-v\in E(G)$, where $[u^\approxeq]\neq [v^\approxeq]$. Let $[G]$ be a simple graph whose vertex set is $[V(G)]$, and edge set is $E([G])$. 

\begin{remark}\label{rem4.5}
	It is easy to observe that, if $G(P)$ be  the  zero-divisor graph, then $[G(P)]=G([P])$.	
\end{remark}

In view of Corollary \ref{corsperfect} and Theorem \ref{redperfect}, we have the following result.

\begin{corollary} [{Bagheri et al. \cite[Corollary 3.2]{bagheri}}] \label{corsqperfect}
	$G$ is perfect if and only if $[G]$ is perfect.
\end{corollary}

\begin{remark}\label{obs2} $G+ I_m$ is a perfect graph if and only if $G$ is a perfect graph if and only if $G\vee K_m$ is a perfect graph. In particular, $G(P)$  is a perfect graph if and only if $G^*(P)$ is a perfect graph.  
	
\end{remark}

With this preparation, we are ready to prove statement $\textbf{(C)}$ of our main result.

\begin{proof}[\underline{\underline{Proof of Theorem \ref{zdgchordal}}}] \textbf{(C)} 	Let $P$ be a finite poset with $n$ atoms such that   $[P]$ is a Boolean  lattice. This implies that $P_{i_1i_2\dots i_k}\neq \emptyset$ for any nonempty set $\{i_1,i_2,\dots, i_k\}\subseteq \{1,2,\dots,n\}$.
	
	Suppose $G(P)$ is a perfect graph. We have to prove that the number $n$ of atoms of $P$ is  $\leq 4$. Suppose on contrary that $n\geq 5$. We show that $G([P])$ contains induced $5$-cycle. For this, we choose the set of vertices $X=\{P_{13}, P_{14}, P_{24}, P_{25}, P_{35}\}$. By Lemma \ref{property}(3), we have the elements of set $X$ induces $5$-cycle $P_{14}- P_{25}- P_{13}- P_{24}- P_{35}- P_{14}$ in $G([P])$. This gives that $G([P])$ is not a perfect graph. By Remark \ref{rem4.5} and  Corollary \ref{corsqperfect}, $G(P)$ is not perfect graph. Thus, $n\leq 4$.

	Conversely, suppose that $n\leq 4$. We have to show that $G(P)$ is perfect. First, we observe that the graph $G(P)$ is perfect for $n\leq 3$. By Theorem \ref{zdgchordal} \textbf{(B)}, $G^{c}(P)$ is a chordal graph , if $n\leq 3$. Thus, by Theorem \ref{choedalperfect}, $G^{c}(P)$ is a perfect graph, if $n\leq 3$. By Corollary \ref{corsperfect}(1), $G(P)$ is a perfect graph, if $n\leq 3$.
	
	Now, we prove that $G(P)$ is a perfect graph, if $n=4$. By Remark \ref{rem4.5} and  Corollary \ref{corsqperfect}, it is enough to prove that $G([P])$ is perfect. Since $n=4$, then $[P]$ is a Boolean lattice with $2^4$ elements, as shown in Figure \ref{b24}(a). The Figure \ref{b24}(b) shows that the zero-divisor graph $G([P])$ of $[P]$. It is not very difficult to verify that, $G([P])$ does not contains an induced odd cycle of length $\geq 5$ and its complement. By Theorem \ref{strongperfect}, $G([P])$ is a perfect graph, if $n=4$. This completes the proof.
	
	\begin{figure}[h]
		\begin{center}
			\begin{tikzpicture}[scale = 1]
				\node [right] at (0,-2.5) {(a) $ \mathbf{ [P]}=\mathbf{2}^4$  };
				
				\draw [fill=black] (1,-1) circle (.05); \node [below] at (1,-1) {$0$};
				\draw [fill=black] (1.5,0) circle (.05); \node [below] at (1.6,0) {$q_3$};
				\draw [fill=black] (2.5,0) circle (.05); \node [below] at (2.6,0) {$q_4$};
				\draw [fill=black] (0.5,0) circle (.05);\node [below] at (0.4,0) {$q_2$}; 
				\draw [fill=black] (-0.5,0) circle (.05); \node [below] at (-.6,0) {$q_1$};
				
				\draw [fill=black] (-1,1.25) circle (.05); \node [left] at (-1,1.25) {$q_{12}$};
				\draw [fill=black] (-0.2,1.25) circle (.05); \node [left] at (-.2,1.25) {$q_{13}$};
				\draw [fill=black] (0.6,1.25) circle (.05); \node [below] at (0.65,1.2) {$q_{14}$};
				\draw [fill=black] (1.4,1.25) circle (.05); \node [right] at (1.4,1.25) {$q_{23}$};
				\draw [fill=black] (2.2,1.25) circle (.05); \node [right] at (2.2,1.25) {$q_{24}$};
				\draw [fill=black] (3,1.25) circle (.05); \node [right] at (3,1.25) {$q_{34}$};
				
				\draw [fill=black] (1.5,2.5) circle (.05); \node [above] at (1.6,2.5) {$q_2^*$};
				\draw [fill=black] (2.5,2.5) circle (.05); \node [above] at (2.5,2.5) {$q_1^*$};
				\draw [fill=black] (0.5,2.5) circle (.05); \node [above] at (.4,2.5) {$q_3^*$};
				\draw [fill=black] (-0.5,2.5) circle (.05); \node [above] at (-.5,2.5) {$q_4^*$};
				\draw [fill=black] (1,3.5) circle (.05);\node [above] at (1,3.5) {$1$}; 
				
				\draw (1,3.5) -- (-0.5,2.5);\draw (1,3.5) -- (0.5,2.5);\draw (1,3.5) -- (1.5,2.5);\draw (1,3.5) -- (2.5,2.5);
				
				\draw (1,-1) -- (-0.5,0);\draw (1,-1) -- (0.5,0);\draw (1,-1) -- (1.5,0);\draw (1,-1) -- (2.5,0);
				
				\draw (-1,1.25) -- (-0.5,0);\draw (-0.2,1.25) -- (-0.5,0);\draw (0.6,1.25) -- (-0.5,0);\draw (-1,1.25) -- (0.5,0);\draw (1.4,1.25) -- (0.5,0);\draw (2.2,1.25) -- (0.5,0);\draw (-0.2,1.25) -- (1.5,0);
				\draw (1.4,1.25) -- (1.5,0);\draw (3,1.25) -- (1.5,0);\draw (3,1.25) -- (2.5,0);\draw (2.2,1.25) -- (2.5,0);\draw (0.6,1.25) -- (2.5,0);
				\draw (-1,1.25) -- (-0.5,2.5);\draw (-0.2,1.25) -- (-0.5,2.5);\draw (1.4,1.25) -- (-0.5,2.5);\draw (-1,1.25) -- (0.5,2.5);\draw (0.6,1.25) -- (0.5,2.5);\draw (2.2,1.25) -- (0.5,2.5);\draw (-.2,1.25) -- (1.5,2.5);\draw (0.6,1.25) -- (1.5,2.5);\draw (3,1.25) -- (1.5,2.5); \draw (1.4,1.25) -- (2.5,2.5);\draw (2.2,1.25) -- (2.5,2.5);\draw (3,1.25) -- (2.5,2.5);

				\begin{scope}[shift={(7,0)}]
					\node [right] at (0.5,-2.5) {(b)  $G(\mathbf{ [P]})$  };
					\draw (0,0) -- (2,0); \draw (0,0) -- (0,2);\draw (0,0) -- (-1,1); \draw (0,0) -- (2,2); \draw (0,0) -- (-0.5,-0.5);\draw (0,0) -- (1,-1);\draw (0,0) -- (3.5,-1.5); 
					\draw (2,0) -- (2.5,-0.5); \draw (2,0) -- (2,2); \draw (2,0) -- (0,2);
					\draw (2,0) -- (3,1); \draw (2,0) -- (3.5,3.5);\draw (2,0) -- (1,-1);
					\draw (2,2) -- (0,2); \draw (2,2) -- (2.5,2.5);\draw (2,2) -- (3,1); \draw (2,2) -- (3.5,-1.5); \draw (2,2) -- (1,3);
					\draw (0,2) -- (-1,1); \draw (0,2) -- (-0.5,2.5);\draw (0,2) -- (1,3); \draw (0,2) -- (3.5,3.5);
					\draw (3.5,-1.5) -- (3.5,3.5); \draw (-1,1) -- (3,1);\draw (1,-1) -- (1,3);

					\draw [fill=black] (0,0) circle (.1); 
					
					\draw [fill=black] (2,0) circle (.1); 
					\draw [fill=black] (2,2) circle (.1); 
					
					\draw [fill=black] (0,2) circle (.1);

					\draw [fill=black] (-0.5,-0.5) circle (.1); \node [below] at (-.5,-.5) {$q_2^*$};
					\draw [fill=black] (-0.5,2.5) circle (.1);\node [above] at (-.5,2.5) {$q_1^*$};
					\draw [fill=black] (2.5,-0.5) circle (.1); \node [below] at (2.5,-.5) {$q_3^*$};
					\draw [fill=black] (2.5,2.5) circle (.1);\node [above] at (2.5,2.5) {$q_4^*$};
					\draw [fill=black] (-1,1) circle (.1);\node [above] at (-1,1) {$q_{34}$};
					\draw [fill=black] (1,3) circle (.1);\node [above] at (1,3) {$q_{23}$};
					\draw [fill=black] (3.5,-1.5) circle (.1);\node [below] at (3.5,-1.5) {$q_{13}$};
					
					\draw [fill=black] (1,-1) circle (.1);\node [below] at (1,-1) {$q_{14}$};
					\draw [fill=black] (3,1) circle (.1);\node [below] at (3.15,1) {$q_{12}$};
					\draw [fill=black] (3.5,3.5) circle (.1);\node [above] at (3.5,3.5) {$q_{24}$};
					
					\node [left] at (-0.05,0) {$q_{2}$};\node [right] at (2.05,0) {$q_{3}$};\node [left] at (-0.05,2) {$q_{1}$};\node [right] at (2.05,2) {$q_{4}$};
				\end{scope}
			\end{tikzpicture}
		\end{center}
		\caption{}\label{b24}
	\end{figure}
	
\end{proof}

\vskip 5truept 

\begin{theorem}[{Joshi \cite[Corollary 2.11]{joshi}}]\label{weaklyperfectatomic}
	Let $G(P)$ be the zero-divisor graph of an atomic poset $P$. Then $\chi(G(P))=\omega(G(P))=n$, where $n$ is the finite number of atoms of $P$.
\end{theorem}

Thus, Theorem \ref{zdgchordal}(C) extends the following result. 

\begin{corollary}[{Patil et al. \cite[Theorem 2.6]{awj}}]\label{t8}
	Let $L$ be a lattice with the least element $0$ such that $G(L)$ is finite. If $G(L)$ is not perfect,  then $\omega(G(L))\geq 5$. Moreover, if $L$ is $0$-distributive, then $G(L)$ contains an induced $5$-cycle.
\end{corollary}

In view of  Proof of Theorem \ref{zdgchordal},  Theorem \ref{gchord},  Theorem \ref{redperfect}, Remark \ref{rem4.5} and  Corollary \ref{corsqperfect}, we have the following result.

\begin{theorem}\label{t3}
	Let $P$ be a finite poset such that   $[P]$ is a Boolean   lattice.
	Then the following statements are equivalent.
	\begin{enumerate} \item $G(P)$
		is perfect. 
		\item $G([P])$
		is perfect. 
		\item The number of atoms of $P$ is $\leq 4$.
		\item $G(P)$ does not contain an induced 5-cycle.
		\item $G([P])$ does not contain an induced 5-cycle.
		\item $[P]$ does not contain a meet sub-semilattice isomorphic to
		the meet sub-semilattice containing the least element 0 of $L$ as depicted in Figure \ref{sc}.
	\end{enumerate}
\end{theorem}

\begin{figure}[h]
	\begin{center}
		\begin{tikzpicture}[scale=.6]
			\draw
			(-1.5,0)--(-.5,2)--(0.5,0)--(1.5,2)--(2.5,0)--(3.5,2)--(4.5,0)--(5.5,2)--(6.5,0)--(7.5,2)--cycle;
			\draw (-1.5,0)--(2.5,-2);\draw (.5,0)--(2.5,-2);\draw
			(2.5,0)--(2.5,-2);\draw (4.5,0)--(2.5,-2);\draw (6.5,0)--(2.5,-2);
			\node [left] at(-1.5,0) { $P_{1}$};\node [above] at(-.5,2)
			{$P_{12}$};\node [left] at(0.6,0) { $P_{2}$};\node [above] at(1.5,2)
			{$P_{23}$};\node [right] at(2.4,0) { $P_{3}$};\node [above]
			at(3.5,2) {$P_{34}$};\node [right] at(4.4,0) { $P_{4}$};\node
			[above] at(5.5,2) {$P_{45}$};\node [right] at(6.4,0) { $P_{5}$};\node [above] at(7.5,2) {$P_{15}$};\node [left] at (2.5,-2.1)
			{$P_{0}$}; \draw[fill=white](-1.5,0) circle(.06);
			\draw[fill=white](-.5,2) circle(.06);\draw[fill=white](.5,0)
			circle(.06);\draw[fill=white](1.5,2)
			circle(.06);\draw[fill=white](2.5,0)
			circle(.06);\draw[fill=white](3.5,2)
			circle(.06);\draw[fill=white](4.5,0)
			circle(.06);\draw[fill=white](5.5,2)
			circle(.06);\draw[fill=white](6.5,0)
			circle(.06);\draw[fill=white](7.5,2)
			circle(.06);\draw[fill=white](2.5,-2) circle(.06);

		\end{tikzpicture}
	\end{center}
	\caption{Meet sub-semilattice}\label{sc}
\end{figure}
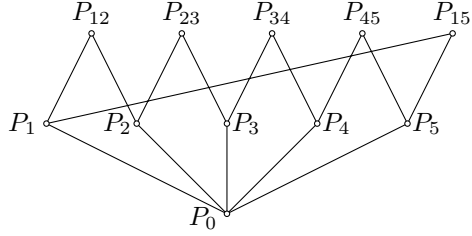

Using Theorem \ref{zdgchordal}, Theorem \ref{nboolean}, Lemma \ref{product} and Theorem \ref{weaklyperfectatomic},   we have the following results.

\begin{corollary}\label{perfect0lattice}
	Let $L$ be a finite $0$-distributive lattice with  $n$ atoms. Then $G(L)$ is perfect if and only if $\omega(G(L))=n\leq 4$. 
\end{corollary}

\begin{corollary} \label{zdgperfectproduct}
	Let $\textbf{P}=\prod\limits_{i=1}^nP^i$, where $P^i$  be a finite bounded poset such that $Z(P^i)=\{0\}$ for every $i$. Then $G(\textbf{P})$ is perfect graph if and only if $\omega(G(\textbf{P}))=n\leq 4$.

\end{corollary}

\begin{remark}
	The condition on poset $P$ with $[P]$ is Boolean lattice  in our first main Theorem \ref{zdgchordal} is necessary. Let $P$ be a uniquely complemented distributive poset with  $5$ atoms and $5$ dual atoms (as shown in Figure \ref{condition}(A)) such that $[P]$ is a Boolean poset but not a lattice. Also, $P_{ij}\not= \emptyset$ for all $i, j \in \{1,2,3,4,5\}$. The zero-divisor graph $G(P)$ of $P$ and its complement $G^c(P)$ are shown in Figure \ref{condition}(B), Figure \ref{condition}(C), respectively. Then $G(P)$ and $G^c(P)$ does not contains an induced  cycle of length $\geq 4$. Thus, $G(P)$ and $G^c(P)$ are the chordal  graphs. Hence, $G(P)$ is a perfect graph though number of atoms $\nleq 4$, contrary to Theorem \ref{zdgchordal}. In fact, we can construct a Boolean poset with  $n$ atoms and $n$ dual atoms ($n$- arbitrarily large) similar to as shown in Figure \ref{condition}(A) such that $G(P)$ and $G^c(P)$ are chordal graphs, hence perfect graphs.

	\begin{figure}[h]
		\begin{center}

			\begin{tikzpicture}[scale =0.85]
				\centering	
				
				\draw [fill=black] (0,0) circle (.1);
				\draw [fill=black] (-2,1.5) circle (.1);
				\draw [fill=black] (-1,1.5) circle (.1);
				\draw [fill=black] (0,1.5) circle (.1);
				\draw [fill=black] (1,1.5) circle (.1);
				\draw [fill=black] (2,1.5) circle (.1);
				\draw [fill=black] (-2,3) circle (.1);
				\draw [fill=black] (-1,3) circle (.1);
				\draw [fill=black] (0,3) circle (.1);
				\draw [fill=black] (1,3) circle (.1);
				\draw [fill=black] (2,3) circle (.1);
				\draw [fill=black] (-2,1.5) circle (.1);
				\draw [fill=black] (-1,1.5) circle (.1);
				\draw [fill=black] (0,1.5) circle (.1);
				\draw [fill=black] (1,1.5) circle (.1);
				\draw [fill=black] (0,4.5) circle (.1);	
				
				\draw (0,0)--(-2,1.5)--(-2,3)--(-1,1.5)--(0,0)--(0,1.5)--(-2,3)--(1,1.5)--(0,0)--(2,1.5)--(-1,3)--(-2,1.5)--(0,3)--(-1,1.5); \draw(1,3)--(0,4.5)--(2,3)--(2,1.5)--(0,3)--(0,4.5)--(-1,3)--(0,1.5)--(1,3)--(2,1.5);	
				
				\draw (-2,3)--(0,4.5)--(2,3)--(1,1.5);
				\draw (-1,1.5)--(-1,3);
				\draw (-2,1.5)--(1,3)--(1,1.5)--(0,3);
				\draw (-1,1.5)--(2,3)--(0,1.5);

				\draw node [above] at (0,4.6) {$1$};
				\draw node [above] at (0,3.1) {$q^*_3$};
				\draw node [above] at (-2,3.1) {$q^*_5$};
				\draw node [above] at (-1,3.1) {$q^*_4$};
				\draw node [above] at (1,3.1) {$q^*_2$};
				\draw node [above] at (2,3.1) {$q^*_1$};
				\draw node [below] at (0,1.4) {$q_3$};
				\draw node [below] at (-2,1.4) {$q_1$};
				\draw node [below] at (-1,1.4) {$q_2$};
				\draw node [below] at (1,1.4) {$q_4$};
				\draw node [below] at (2,1.4) {$q_5$};
				\draw node [below] at (0,-0.1) {$0$};
				
				\draw node [above] at (0,-1.5) {A. $P$};

				\begin{scope}[shift={(4,1.)}]
					
					\draw [fill=black] (0,0) circle (.1);
					\draw [fill=black] (0,2) circle (.1);
					\draw [fill=black] (2.5,0) circle (.1);
					\draw [fill=black] (2.5,2) circle (.1);
					\draw [fill=black] (1.25,3.5) circle (.1);
					\draw [fill=black] (-0.75,3) circle (.1);
					\draw [fill=black] (3.25,3.3) circle (.1);
					\draw [fill=black] (-1.5,1) circle (.1);
					\draw [fill=black] (4,1) circle (.1);
					\draw [fill=black] (1.25,-1.5) circle (.1);
					
					\draw (0,0)--(0,2)--(2.5,2)--(2.5,0)--(0,0)--(1.25,3.5)--(2.5,2)--(0,0)--(1.25,-1.5);
					
					\draw (2.5,0)--(0,2)--(1.25,3.5)--(2.5,0)--(4,1);
					
					\draw (0,2)--(-1.5,1);
					
					\draw (1.25,3.5)--(-0.75,3);
					
					\draw (3.25,3.3)--(2.5,2);
					
					\draw node [below] at (0,0) {$q_1$};
					\draw node [below] at (2.5,0) {$q_2$};
					\draw node [above] at (2.5,2) {$q_3$};
					\draw node [above] at (1.25,3.5) {$q_4$};
					\draw node [above] at (0,2) {$q_5$};
					
					\draw node [above] at (-1.5,1) {$q^*_5$};
					\draw node [above] at (1.25,-1.5) {$q^*_1$};
					\draw node [above] at (4,1) {$q^*_2$};
					\draw node [above] at (3.25,3.3) {$q^*_3$};
					\draw node [above] at (-0.75,3) {$q^*_4$};
					
					\draw node [above] at (1.25,-2.5) {B. $G(P)$};
					
				\end{scope}

				\begin{scope}[shift={(10,1)}]
					
					\draw [fill=black] (0,0) circle (.1);
					\draw [fill=black] (0,2) circle (.1);
					\draw [fill=black] (2.5,0) circle (.1);
					\draw [fill=black] (2.5,2) circle (.1);
					\draw [fill=black] (1.25,3.5) circle (.1);
					\draw [fill=black] (-0.75,3) circle (.1);
					\draw [fill=black] (3.25,3) circle (.1);
					\draw [fill=black] (-1.5,1) circle (.1);
					\draw [fill=black] (4,1) circle (.1);
					\draw [fill=black] (1.25,-1.5) circle (.1);
					
					\draw (0,0)--(0,2)--(2.5,2)--(2.5,0)--(0,0)--(1.25,3.5)--(2.5,2)--(0,0)--(1.25,-1.5)--(2.5,0)--(0,2)--(1.25,3.5)--(3.25,3)--(2.5,2)--(4,1)--(2.5,0)--(1.25,-1.5)--(0,0)--(-1.5,1)--(2.5,0)--(3.25,3)--(0,2)--(-0.75,3)--(2.5,2);
					
					\draw (-1.5,1)--(0,2)--(1.25,-1.5)--(2.5,2)--(4,1)--(0,0)--(-0.75,3)--(1.25,3.5)--(2.5,0);
					
					\draw (-1.5,1)--(1.25,3.5)--(4,1);
					
					\draw node [below] at (0,0) {$q^*_1$};
					\draw node [below] at (2.5,0) {$q^*_2$};
					\draw node [above] at (2.5,2) {$q^*_3$};
					\draw node [above] at (1.25,3.5) {$q^*_4$};
					\draw node [above] at (0,2) {$q^*_5$};

					\draw node [above] at (-1.5,1) {$q_3$};
					\draw node [above] at (1.25,-1.5) {$q_4$};
					\draw node [above] at (4,1) {$q_5$};
					\draw node [above] at (3.25,3.1) {$q_1$};
					\draw node [above] at (-0.75,3) {$q_2$};
					
					\draw node [above] at (1.25,-2.5) {C. $G^c(P)$};
					
				\end{scope}
				
			\end{tikzpicture}
		\end{center}	
		\caption{}\label{condition}
	\end{figure}
\end{remark}

\newpage

\section{Coloring of zero-divisor graphs}

In this section, we prove that the zero-divisor graphs of finite posets satisfy the Total Coloring Conjecture. Further, we prove that the complement of the zero-divisor graph of finite 0-distributive posets satisfies the Total Coloring Conjecture. Recently, Srinivasa Murthy \cite[Revision 3]{murthy} claims the proof of Total Coloring Conjecture for the finite simple graphs. However, this claim is not verified.

\vskip 5truept 
We quote the following definition and results needed in the sequel.

\vskip 5truept 

The \textit{vertex chromatic number} $\chi(G)$ of a graph $G$ is the minimum number of colors required to color the vertices of $G$ such that no two adjacent vertices receive the same color, whereas the \textit{edge chromatic number} $\chi'(G)$ of  $G$ is the minimum number of colors required to color the edges of $G$ such that incident edges receive different colors. A graph $G$ is {\it class one}, if $\chi'(G)=\Delta(G)$  and is {\it class two},	if $\chi'(G)=\Delta(G)+1$.

Note that every graph $G$ requires at least $\Delta(G)+1$ colors for the total coloring. A graph is said to be  {\it type I}, if $\chi''(G) =\Delta(G)+1$ and is said to be {\it type II}, if $\chi''(G) =\Delta(G)+2$.


\begin{theorem}[Behzad et al. {\cite[Theorem 1]{behzad}}]\label{complete} The following statements hold for the complete graph $K_n$.
	\begin{enumerate}
		\item $\chi(K_{n})=n$.
		\item $\chi{'}(K_{n})= \begin{cases}
			n & \text{ for $n$ odd  $n\geq 3;$} \\ 
			n-1 & \text{for $n$ even.} 
		\end{cases}$
		
		\item $\chi{''}(K_{n})= \begin{cases}
			n  & \text{for $n$ odd;}\\ n+1 & \text{ for $n$ even.} \end{cases}$
	\end{enumerate}
\end{theorem}
\vskip 5truept 
\begin{theorem}[Behzad et al. {\cite[Theorem 2]{behzad}}]\label{bipartite} The following statements hold for the complete bipartite graph $K_{m,n}$.
	\begin{enumerate}
		
		\item $\chi(K_{m,n})=2$;
		\item $\chi{'}(K_{m,n})=$ max$\{m,n\}$;
		\item $\chi{''}(K_{m,n})=$ max$\{m,n\}+1+\delta_{mn}$, where $\delta_{mn}= \begin{cases} 0 & \text{if $m\neq n;$}\\ 1 & \text{if $m= n.$} \end{cases}$ \end{enumerate}
\end{theorem}
\vskip 5truept

\begin{theorem}[Yap {\cite[Theorem 2.6, p. 8]{tcyap}}]\label{zdg}
	Let $G$ be a graph of order $N$ and $\Delta(G)$ be the maximum degree of graph $G$. If $G$ contains an independent set $X$ of vertices, where $|X|\geq N-\Delta(G)-1$, then $\chi''(G)\leq\Delta(G)+2$, that is, $G$ satisfies the Total Coloring Conjecture.
\end{theorem}

\begin{theorem}[Yap {\cite[Theorem 5.15, p. 52]{tcyap}}]\label{czdg}
	If $G$ is a graph having $\Delta(G)\geq\frac{3}{4}|G|$, then \linebreak $\chi''(G)\leq\Delta(G)+2$, that is, $G$ satisfies the Total Coloring Conjecture.
\end{theorem}

\begin{theorem}[Khandekar and Joshi {\cite[Theorem 1.2, p.1]{nkvj}}]\label{tcc} Let  $ \mathbf{ P}=\prod\limits_{i=1}^{n}P^i$, $(n\geq 2)$, where $P^i$'s are  finite bounded posets such that $Z(P^i)=\{0\}$, $\forall i$ and $2\leq|P^1|\leq|P^2|\leq\dots\leq|P^n|$. Then $G(\mathbf{ P})$ satisfies the Total Coloring Conjecture. In particular, \\
	$\chi{''}(G(\mathbf{ P}))=\begin{cases}  \Delta(G(\mathbf{ P}))+2 & \text{if $n=2$ and $|P^1|=|P^2|;$}\\ \Delta(G(\mathbf{ P}))+1 & \text{ otherwise.}\\ \end{cases}$
\end{theorem}

\begin{theorem}[Khandekar and Joshi {\cite[Theorem 3.5, p.6]{nkvj}}]\label{edge chromatic number} Let  $ \mathbf{ P}=\prod\limits_{i=1}^{n}P^i$, $(n\geq 2)$, where $P^i$'s are  finite bounded posets such that $Z(P^i)=\{0\}$, $\forall i$ and $2\leq|P^1|\leq|P^2|\leq\dots\leq|P^n|$. Let $G(\mathbf{ P})$ be the zero-divisor graph of a poset $\mathbf{ P}$.  Then 	$\chi{'}(G(\mathbf{ P}))=\Delta(G(\mathbf{ P}))$.
\end{theorem}

\begin{remark} [Khandekar and Joshi {\cite[Remark 3.9, p.13]{nkvj}}]\label{type2} Let  $ \mathbf{ P}=\prod\limits_{i=1}^{n}P^i$, $(n\geq 2)$, where $P^i$'s are  finite bounded posets such that $Z(P^i)=\{0\}$, $\forall i$ and $2\leq|P^1|\leq|P^2|\leq\dots\leq|P^n|$. 
	The zero-divisor graph of $\mathbf{P}$ is of type II if and only if  $\mathbf{P}$ is a direct product of exactly two posets $P^1$, $P^2$ with $Z(P^i)=\{0\}$ for $i=1,2$ and  $|P^1|=|P^2|$. 
\end{remark}

With this preparation, we are ready to prove our second main result of the paper which follows from  Theorem \ref{zdgtcc} and Theorem \ref{czdgtcc}.

\begin{theorem}\label{zdgtcc}
	Let $P$ be a finite poset. Then $G(P)$ satisfies the Total Coloring Conjecture.
\end{theorem}

\begin{proof}
	Let $q_1, q_2, \dots, q_n$ be  all atoms of  $P$. Let $G(P)$ be the zero-divisor graph of $P$ of order $N$, i.e., $|V(G(P))|=N$ and $\Delta(G(P))$ be the maximum degree of  $G(P)$. Then $q_i^u\cap V(G(P))$ is an independent set of  $G(P)$, for every $i\in\{1,\dots,n\}$. Hence $\alpha(G(P))\geq |q_i^u\cap V(G(P))|$. It is clear to see that, if $x\in V(G(P))$ and $x\notin q_i^u$, then $x$, $q_i$ are adjacent in $G(P)$. Therefore, the degree of $q_i$ is equal to $|V(G(P))\setminus q_i^u|$, for every $i\in\{1,\dots,n\}$. Hence, $\Delta(G)\geq$ deg$(q_i)$, for all $i$. Note that if $x\in V(G(P))$, then either $x\in q_i^u\cap V(G(P))$ or $x\in V(G(P))\setminus q_i^u$. Thus, $V(G(P))$ is the disjoint union of $q_i^u\cap V(G(P))$ and $V(G(P))\setminus q_i^u$, for all $i$. In particular, $V(G(P))$ is the disjoint union of $q_1^u\cap V(G(P))$ and $V(G(P))\setminus q_1^u$. Let $|q_1^u\cap V(G(P))|=m$ and $|V(G(P))\setminus q_1^u|=k$. We have $m+k=N$. From this, we get $m\geq N-k-1$. This implies that $m\geq N-\Delta(G(P))-1$ (as $\Delta(G)\geq$ deg$(q_1)$). By Theorem \ref{zdg}, $G(P)$ satisfies the Total Coloring Conjecture.	\end{proof}

Now, we prove that if $P$ is a finite 0-distributive poset, then $G^c(P)$ satisfies Total Coloring Conjecture. Since the proof of this result is a little lengthy, we first give the skeleton of it using  Figure \ref{example}. Consider the $0$-distributive poset $P$ and it's zero-divisor graph shown in Figure  \ref{example}(a) and (b). This is symbolically shown in Figure \ref{example}(c), where $I_4$ is the independent set on $4$ vertices and the edges between $I_4$ is nothing but the induced graph $I_4\vee I_4$. One can compare this graph with Figure \ref{fig1}(a). With the same idea, the complement of zero-divisor graph of $P$ is shown in Figure \ref{example}(d) (see also Figure \ref{fig1}(b)). 

Now, we discuss the skeleton of the proof of Theorem \ref{czdgtcc}. If a poset has exactly two atoms, then $G^c(P)$ is a union of two complete graphs and hence satisfies the Conjecture. Further, if $n\geq 4$, then we use Theorem \ref{czdg} to prove that $G^c(P)$ satisfies the Total Coloring Conjecture. The case when $n=3$ is a little complicated. 

First, we do the total coloring to the vertex induced subgraph $G'$ of $G^c(P)$ on the vertices of $[q_1]^*,[q_2]^*,[q_3]^*,[q_3]$, where $|[q_i]|=l_i$ and   $|[q_i]^*| =m_i$. We denote $[q_i]=\{q_{i1},q_{i2},\dots,q_{il_i}\}$ and $[q_i]^*=\{q_{i1}^*,q_{i2}^*,\dots,q_{im_i}^*\}$. Without loss of generality,   assume that $l_3\geq l_2\geq l_1$.
Then we do the total coloring to the vertex induced subgraph by $[q_2]$ and followed by $[q_1]$. Afterwards, we do the edge coloring of the complete bipartite graph $K_{|[q_3]*|} \vee K_{|[q_1]|}$. We denote the set $V_1=\{q_{21},q_{22},\dots,q_{2l_2}\}$ and $V_2=\{q_{11}^*,\dots,q_{1m_1}^*,~ q_{31}^*,\dots,q_{3m_3}^*\}$.  Let us complete the edge coloring of the complete bipartite graph $K_{s,t}$ on  sets $V_1$ and $V_2$, where $s=l_2, ~ t=m_1+m_3$. 
Similarly, we do the edge coloring of the complete bipartite graph $K'_{r',s'}$ on  sets $V_1'$ and $V_2'$, where $r'=l_1,~ s'=m_2$,~ $V_1'=\{q_{11},q_{12},\dots,q_{1l_1}\}$ and $V_2'=\{q_{21}^*,\dots,q_{2m_2}^*\}$.

In conclusion, we prove that this is a required total coloring of $G^c(P)$.

\begin{figure}[h]
	\begin{center}
		\begin{tikzpicture}[scale =1]
			
			\draw[fill=pink] (-1.5,1.8) ellipse(0.2 and 1.2);
			\draw[fill=pink] (0,1.8) ellipse(0.2 and 1.2);
			\draw[fill=pink] (1.5,1.8) ellipse(0.2 and 1.2);
			
			\draw[fill=aquamarine] (-1.5,4.8) ellipse(0.2 and 1.2);
			\draw[fill=aquamarine] (0,4.8) ellipse(0.2 and 1.2);
			\draw[fill=aquamarine] (1.5,4.8) ellipse(0.2 and 1.2);
			
			\node [left] at (-1.6,1.8) {$[q_1]$}; \node [left] at (-0.1,1.8) {$[q_2]$}; \node [right] at (1.6,1.8) {$[q_3]$};
			\node [left] at (-1.6,4.8) {$[q_3]^*$}; \node [left] at (-0.1,4.8) {$[q_2]^*$}; \node [right] at (1.6,4.8) {$[q_1]^*$};

			\draw [fill=black] (0,-0.5) circle (0.08);\draw [fill=black] (0,7) circle (0.08);
			
			\draw [fill=black] (0,1) circle (0.08);\draw [fill=black] (0,1.5) circle (0.08);\draw [fill=black] (0,2.5) circle (0.08);\draw [fill=black] (0,2) circle (0.08);
			
			\draw [fill=black] (-1.5,5.5) circle (0.08);\draw [fill=black] (-1.5,4) circle (0.08);\draw [fill=black] (-1.5,4.5) circle (0.08);\draw [fill=black] (-1.5,5) circle (0.08);
			
			\draw [fill=black] (-1.5,1) circle (0.08);\draw [fill=black] (-1.5,1.5) circle (0.08);\draw [fill=black] (-1.5,2.5) circle (0.08);\draw [fill=black] (-1.5,2) circle (0.08);
			
			\draw [fill=black] (1.5,5.5) circle (0.08);\draw [fill=black] (1.5,4) circle (0.08);\draw [fill=black] (1.5,4.5) circle (0.08);\draw [fill=black] (1.5,5) circle (0.08);
			
			\draw [fill=black] (1.5,1) circle (0.08);\draw [fill=black] (1.5,1.5) circle (0.08);\draw [fill=black] (1.5,2.5) circle (0.08);\draw [fill=black] (1.5,2) circle (0.08);
			
			\draw [fill=black] (1.5,5.5) circle (0.08);\draw [fill=black] (1.5,4) circle (0.08);\draw [fill=black] (1.5,4.5) circle (0.08);\draw [fill=black] (1.5,5) circle (0.08);
			
			\draw [fill=black] (0,5.5) circle (0.08);\draw [fill=black] (0,4) circle (0.08);\draw [fill=black] (0,4.5) circle (0.08);\draw [fill=black] (0,5) circle (0.08);
			
			\draw (0,-0.5) -- (-1.5,1);\draw (0,-0.5) -- (0,1);\draw (0,-0.5) -- (1.5,1);
			
			\draw (-1.5,5.5) -- (-1.5,1);\draw (0,2.5) -- (0,1);\draw (1.5,5.5) -- (1.5,1);
			
			\draw (-1.5,2.5) -- (0,4);\draw (0,2.5) -- (0,1);\draw (1.5,2.5) -- (1.5,1);\draw (1.5,2.5) -- (0,4);\draw (0,2.5) -- (1.5,4);
			\draw (0,2.5) -- (-1.5,4); \draw (0,5.5) -- (0,4);\draw (0,7) -- (0,4); \draw (0,7) -- (-1.5,5.5);\draw (0,7) -- (1.5,5.5);
			
			\node [above] at (0,7) {$1$};
			\node [below] at (0,-0.6) {$0$};
			\node [below] at (-0.2,-1) {$(a)~ P$};

			\begin{scope}[shift={(7,2.5)}]
				
				\node [left] at (2.5,0) {$[q_1]$};
				\node [left] at (2.7,-2) {$[q_1]^*$}; \node [left] at (2,3.8) {$[q_2]$}; \node [left] at (-0.2,3.8) {$[q_3]$}; \node [left] at (4.8,3.8) {$[q_2]^*$}; \node [left] at (-3.3,3.8) {$[q_3]^*$};

				\draw[fill=pink] (0.7,0) ellipse(1 and 0.5);
				\draw[fill=aquamarine] (0.7,-2) ellipse(1 and 0.5);
				\begin{scope}[rotate=50]{\tiny
						\draw[fill=pink] (1.3,2.5) ellipse(1 and 0.5); \draw[fill=aquamarine] (1.,4.45) ellipse(1 and 0.5);}\end{scope}
				\begin{scope}[rotate=125]{\tiny
						\draw[fill=pink] (1.,-3.2) ellipse(1 and 0.5); \draw[fill=aquamarine] (1.,-4.8) ellipse(1 and 0.5);}\end{scope}
				
				\draw [fill=black] (0,0) circle (.05);\draw [fill=black] (0.5,0) circle (.05);\draw [fill=black] (1,0) circle (.05);\draw [fill=black] (1.5,0) circle (.05);
				
				\draw [fill=black] (0,-2) circle (.05);\draw [fill=black] (0.5,-2) circle (.05);\draw [fill=black] (1,-2) circle (.05);\draw [fill=black] (1.5,-2) circle (.05);
				
				\draw [fill=black] (-1.5,2) circle (.05); \draw [fill=black] (-1.2,2.4) circle (.05);\draw [fill=black] (-0.9,2.8) circle (.05);\draw [fill=black] (-0.6,3.2) circle (.05);
				
				\draw [fill=black] (-3.3,3) circle (.05); \draw [fill=black] (-3.0,3.4) circle (.05);\draw [fill=black] (-2.7,3.8) circle (.05);\draw [fill=black] (-2.4,4.2) circle (.05);
				
				\draw [fill=black] (2.5,2) circle (.05); \draw [fill=black] (2.2,2.4) circle (.05);\draw [fill=black] (1.9,2.8) circle (.05);\draw [fill=black] (1.6,3.2) circle (.05);
				\draw [fill=black] (3.8,3) circle (.05); \draw [fill=black] (3.5,3.4) circle (.05);\draw [fill=black] (3.2,3.8) circle (.05);\draw [fill=black] (2.9,4.2) circle (.05);
				
				\draw (0,0) -- (0,-2);\draw (0.5,0) -- (0,-2);\draw (1,0) -- (0,-2);\draw (1.5,0) -- (0,-2); \draw (0,0) -- (0.5,-2);\draw (0.5,0) -- (0.5,-2);\draw (1,0) -- (0.5,-2);\draw (1.5,0) -- (0.5,-2); \draw (0,0) -- (1,-2);\draw (0.5,0) -- (1,-2);\draw (1,0) -- (1,-2);\draw (1.5,0) -- (1,-2);
				\draw (0,0) -- (1.5,-2);\draw (0.5,0) -- (1.5,-2);\draw (1,0) -- (1.5,-2);\draw (1.5,0) -- (1.5,-2);

				\draw (2.5,2) -- (3.8,3);\draw (0.5,0) -- (2.5,2);\draw (1,0) -- (2.5,2);\draw (1.5,0) -- (2.5,2);
				
				\draw (0,0) -- (2.2,2.4);\draw (0.5,0) -- (2.2,2.4);\draw (1,0) -- (2.2,2.4);\draw (1.5,0) -- (2.2,2.4);
				
				\draw (0,0) -- (1.9,2.8);\draw (0.5,0) -- (1.9,2.8);\draw (1,0) -- (1.9,2.8);\draw (1.5,0) -- (1.9,2.8);
				
				\draw (0,0) -- (1.6,3.2);\draw (0.5,0) -- (1.6,3.2);\draw (1,0) -- (1.6,3.2);\draw (1.5,0) -- (1.6,3.2);

				\draw (0,0) -- (2.5,2);\draw (3.5,3.4) -- (2.5,2);\draw (3.2,3.8) -- (2.5,2);\draw (2.9,4.2) -- (2.5,2);
				
				\draw (3.8,3) -- (2.2,2.4);\draw (3.5,3.4) -- (2.2,2.4);\draw (3.2,3.8) -- (2.2,2.4);\draw (2.9,4.2) -- (2.2,2.4);
				
				\draw (3.8,3) -- (1.9,2.8);\draw (3.5,3.4) -- (1.9,2.8);\draw (3.2,3.8) -- (1.9,2.8);\draw (2.9,4.2) -- (1.9,2.8);
				
				\draw (3.8,3) -- (1.6,3.2);\draw (3.5,3.4) -- (1.6,3.2);\draw (3.2,3.8) -- (1.6,3.2);\draw (2.9,4.2) -- (1.6,3.2);

				\draw (-1.5,2) -- (2.5,2);\draw (-1.2,2.4) -- (2.5,2);\draw (-0.9,2.8) -- (2.5,2);\draw (-0.6,3.2) -- (2.5,2);
				
				\draw (-1.5,2) -- (2.2,2.4);\draw (-1.2,2.4) -- (2.2,2.4);\draw (-0.9,2.8)  -- (2.2,2.4);\draw (-0.6,3.2) -- (2.2,2.4);
				
				\draw (-1.5,2) -- (1.9,2.8);\draw (-1.2,2.4) -- (1.9,2.8);\draw (-0.9,2.8)  -- (1.9,2.8);\draw (-0.6,3.2) -- (1.9,2.8);
				
				\draw (-1.5,2) -- (1.6,3.2);\draw (-1.2,2.4) -- (1.6,3.2);\draw (-0.9,2.8)  -- (1.6,3.2);\draw (-0.6,3.2) -- (1.6,3.2);

				\draw (-1.5,2) -- (0,0);\draw (-1.2,2.4) -- (0,0);\draw (-0.9,2.8) -- (0,0);\draw (-0.6,3.2) -- (0,0);
				
				\draw (-1.5,2) -- (0.5,0);\draw (-1.2,2.4) -- (0.5,0);\draw (-0.9,2.8)  -- (0.5,0);\draw (-0.6,3.2) -- (0.5,0);
				
				\draw (-1.5,2) -- (1,0);\draw (-1.2,2.4) -- (1,0);\draw (-0.9,2.8)  -- (1,0);\draw (-0.6,3.2) -- (1,0);
				
				\draw (-1.5,2) -- (1.5,0);\draw (-1.2,2.4) -- (1.5,0);\draw (-0.9,2.8)  -- (1.5,0);\draw (-0.6,3.2) -- (1.5,0);
				
				\draw (-1.5,2) -- (-3.3,3);\draw (-1.2,2.4) -- (-3.3,3);\draw (-0.9,2.8) -- (-3.3,3);\draw (-0.6,3.2) -- (-3.3,3);
				
				\draw (-1.5,2) -- (-3,3.4);\draw (-1.2,2.4) -- (-3,3.4);\draw (-0.9,2.8)  -- (-3,3.4);\draw (-0.6,3.2) -- (-3,3.4);
				
				\draw (-1.5,2) -- (-2.7,3.8);\draw (-1.2,2.4) -- (-2.7,3.8);\draw (-0.9,2.8)  -- (-2.7,3.8);\draw (-0.6,3.2) -- (-2.7,3.8);
				
				\draw (-1.5,2) -- (-2.4,4.2);\draw (-1.2,2.4) -- (-2.4,4.2);\draw (-0.9,2.8)  -- (-2.4,4.2);\draw (-0.6,3.2) -- (-2.4,4.2);

				\node [below] at (0.5,-3.2) {$(b)~ G(P)$};
				
			\end{scope}


				\begin{scope}[shift={(0,-6)}]
					
					\draw [fill=pink] (0,0) circle (0.5);\draw [fill=pink] (-1.5,2) circle (0.5);\draw [fill=pink] (1.5,2) circle (0.5);\draw [fill=aquamarine] (0,-2) circle (0.5);\draw [fill=aquamarine] (3,3) circle (0.5);\draw [fill=aquamarine] (-3,3) circle (0.5);
					
					\draw (-0.25,-0.4) -- (-0.25,-1.6);\draw (0.25,-0.4) -- (0.25,-1.6);\draw (0.5,0.1) -- (1.6,1.5);\draw (0.1,0.5) -- (1.1,1.7);\draw (-0.5,0.1) -- (-1.6,1.5);\draw (-0.1,0.5) -- (-1.1,1.7);
					\draw (-1.05,1.8) -- (1.05,1.8);\draw (-1.2,2.4) -- (1.2,2.4);
					\draw (-2,1.85) -- (-2.9,2.5);\draw (-1.8,2.4) -- (-2.5,2.9);
					\draw (2,1.85) -- (2.9,2.5);\draw (1.8,2.4) -- (2.5,2.9);
					
					\node [right] at (-.5,0) {$I_{|[q_1]|}$};
					\node [right] at (-.6,-2) {$I_{|[q_1]^*|}$};\node [right] at (-2,2) {$I_{|[q_3]|}$};\node [right] at (1.,2) {$I_{|[q_2]|}$}; \node [right] at (2.4,3) {$I_{|[q_2]^*|}$};\node [right] at (-3.6,3) {$I_{|[q_3]^*|}$};
					
					\node [right] at (-1.05,2.05) {$I_{|[q_3]|}\vee I_{|[q_2]|}$};
					
					\node [below] at (0.5,-3.5) {$(c)~ G(P)$};
				\end{scope}

				\begin{scope}[shift={(8,-6.5)}]
					
					\draw [fill=aquamarine] (0,4) circle (0.6);
					\draw [fill=aquamarine] (-2,0) circle (0.6);\draw [fill=aquamarine] (2,0) circle (0.6);\draw [fill=pink] (2.8,3.1) circle (0.6);\draw [fill=pink] (-2.8,3.1) circle (0.6);\draw [fill=pink] (0,-2) circle (0.6);
					\node [right] at (-.7,4) {$K_{|[q_3]^*|}$};\node [right] at (-2.65,0) {$K_{|[q_1]^*|}$};\node [right] at (1.4,0) {$K_{|[q_2]^*|}$};\node [right] at (2.2,3.1) {$K_{|[q_1]|}$};\node [right] at (-3.4,3.1) {$K_{|[q_2]|}$};\node [right] at (-.6,-2) {$K_{|[q_3]|}$};
					\node [right] at (-1.35,-0.05) {$K_{|[q_1]^*|}\vee K_{|[q_2]^*|}$};
					
					\draw (-1.5,0.3) -- (1.5,0.3);\draw (-1.5,-0.3) -- (1.5,-0.3);
					\draw (-1.55,0.35) -- (-0.05,3.4);\draw (1.55,0.35) -- (0.05,3.4);\draw (-2.1,0.55) -- (-0.45,3.6);\draw (2.1,0.55) -- (0.45,3.6); 
					\draw (-2.15,0.55) -- (-2.5,2.6);\draw (2.15,0.55) -- (2.5,2.6);
					\draw (-2.6,0.05) -- (-3.,2.55);\draw (2.6,0.05) -- (3.,2.55);
					
					\draw (-2.2,3.1) -- (-.6,3.85);\draw (-2.6,3.65) -- (-.2,4.6);\draw (2.2,3.1) -- (.6,3.85);\draw (2.6,3.65) -- (.2,4.6);
					
					\draw (-1.65,-.45) -- (-.45,-1.65);\draw (-2.3,-.5) -- (-.5,-2.3);\draw (1.65,-.45) -- (.45,-1.65);\draw (2.3,-.5) -- (.5,-2.3);
					
					\node [below] at (0.5,-3.0) {$(d)~ G^c(P)$};
				\end{scope}
				
			\end{tikzpicture}
			
		\end{center}
		\caption{}\label{example}
	\end{figure}
	
	\begin{theorem}\label{czdgtcc}
		Let $P$ be a finite $0$-distributive  poset. Then $G^c(P)$ satisfies the Total Coloring Conjecture.
	\end{theorem} 
	
	\begin{proof}
		
		Let $q_1, q_2, \dots, q_n$ be all atoms of $P$. Then using the statements (1) and (7) of Lemma \ref{property}, $[q_1],[q_2],\dots,[q_n]$ are $n$ atoms of $[P]$ and $[q_i]^*$ is the pseudocomplement of $[q_i]$ in $[P]$ for every $i\in\{1,2,\dots,n\}$. Clearly,  $[q_k]^*$ is a pendent vertex of $G([P])$ adjacent to $[q_k]$ only (see Figure \ref{fig1}(a), if $n=3$).	Let $G^c(P)$ be the complement of the zero-divisor graph of $P$ and let $G^c([P])$ be the complement of the  zero-divisor graph of $[P]$.
		
		
		To get  $\Delta (G^c(P))$, we calculate $\delta(G(P))$. We claim that $\delta(G(P))= $ deg$(x)=|[q_i]|$ for some $i\in \{1,\dots, n\}$ and $x\in [q_i]^*$.

		If $[a]\in V(G([P]))\setminus\{[q_1]^*,\dots,[q_n]^*\}$, then $[a]$ is adjacent to at least two $[q_i]$ in $G([P])$ for $n\geq 3$. Otherwise, if $[a]$ is adjacent to exactly one $[q_i]$ in $G([P])$, then $[a]=[q_i]^*$, a contradiction. Hence, we assume that $[a]$ is adjacent of $[q_i]$ and $[q_j]$ for  $i, j$ in $G([P])$.   Hence if $y\in [a]$, then $y$ will be adjacent to every vertex of $[q_i]$ and $[q_j]$. Hence deg$_{G(P)}(y)\geq|[q_i]|+|[q_j]|$ for  $i,j\in\{1,\dots,n\}$.	Since   $[q_i]^*$ is adjacent to  $[q_i]$ only (Lemma \ref{property}(7)) in $G([P])$, by Lemma \ref{property}(4),  it is easy to prove that deg$_{G(P)}(x)=|[q_i]|$ for $x\in [q_i]^*$. This together with deg$_{G(P)}(y)\geq|[q_i]|+|[q_j]|$ gives that $\delta(G(P))=$deg$(x)=|[q_i]|$, for some $i\in \{1,\dots, n\}$ and $x\in [q_i]^*$. Therefore,\\

		\underline{\underline{ $\Delta(G^c(P))=$ deg$_{G^c(P)}(x)=|V(G^c(P))|-|[q_i]|-1$ for some $i\in \{1,\dots, n\}$ and $x\in [q_i]^*$-------- (A) 
		}}	
		\vskip 5truept 
		We  consider the following cases on $n$:
		
		\noindent 	\textbf{Case-I:} $n=2$.\\
		If $n=2$, then $[P]$ is a  poset with exactly two atoms $[q_1], [q_2]$ with $|[q_i]|=m_i$. In this case, the complement of the zero-divisor graph $G^c([P])$ of $[P]$ is a graph having two isolated vertices. Therefore $G^c(P)$ is the graph $K_{m_{1}} + K_{m_{2}}$. Clearly, $G^c(P)$ satisfies the Total Coloring Conjecture.

		\vskip 5truept 
		\noindent 	\textbf{Case-II:} $n=3$.
		
		Now, assume that  $n=3$. Then $[P]$ is a  poset with three atoms, $[q_1],[q_2]$ and $[q_3]$. Then $V(G([P]))=\{[q_1],[q_2],[q_3],[q_1]^*,[q_2]^*,[q_3]^*\}$. The zero-divisor graph $G([P])$ of $[P]$ and its complement  $G^c([P])$ are shown in Figure \ref{fig1}.

		\begin{figure}[h]
			\begin{center}
				\begin{tikzpicture}[scale =0.5]

					\begin{scope}[shift={(6,2)}]
						
						\draw (0,0) -- (1.5,2);\draw (0,0) -- (-1.5,-1.5); \draw (0,0) -- (3,0);\draw (1.5,4) -- (1.5,2);\draw (3,0) -- (1.5,2); \draw (3,0) -- (4.5,-1.5);
						
						\draw [fill=black] (0,0) circle (.1);\draw [fill=black] (3,0) circle (.1);\draw [fill=black] (1.5,2) circle (.1);\draw [fill=black] (1.5,4) circle (.1);\draw [fill=black] (-1.5,-1.5) circle (.1); \draw [fill=black] (4.5,-1.5) circle (.1);
						
						\node [above] at (-0.3,0) {$[q_3]$}; \node [above] at (3.2,0) {$[q_2]$};\node [above] at (1.5,4.1) {$[q_1]^*$};\node [above] at (-1.5,-1.4) {$[q_3]^*$};\node [above] at (4.6,-1.4) {$[q_2]^*$};\node [right] at (1.6,2) {$[q_1]$};
						\node [below] at (1.85,-3.2) {(a) $G([P])$}; 
					\end{scope}

					\begin{scope}[shift={(14,3)}]
						\draw [fill=black] (0,0) circle (.1);\draw [fill=black] (4,0) circle (.1);\draw [fill=black] (2,3) circle (.1);\draw [fill=black] (2,-1.75) circle (.1);\draw [fill=black] (-0.2,2.6) circle (.1);\draw [fill=black] (4.2,2.6) circle (.1);
						
						\node [left] at (0,0) {$[q_3]^*$};\node [right] at (4,0) {$[q_2]^*$};\node [above] at (2,3) {$[q_1]^*$};\node [above] at (-0.2,2.7) {$[q_2]$};\node [above] at (4.2,2.7) {$[q_3]$};\node [below] at (2,-1.8) {$[q_1]$};

						\draw (0,0) -- (4,0);\draw (0,0) -- (2,3);\draw (0,0) -- (2,-1.75);\draw (0,0) -- (-0.2,2.6);\draw (2,-1.75) -- (4,0);\draw (2,3) -- (4,0);\draw (4.2,2.6) -- (4,0);\draw (2,3) -- (-0.2,2.6);\draw (2,3) -- (4.2,2.6);\node [below] at (1.85,-4.2) {(b) $G^c([P])$};
						
					\end{scope}
					
				\end{tikzpicture}
			\end{center}
			\caption{}\label{fig1}
		\end{figure}
		
		Since deg$_{G^c(P)}(x)=|V(G^c(P))|-|[q_i]|-1$ for  $x\in [q_i]^*$ for some $i$, by $(A)$, and $\Delta(G^c(P))\geq$ deg$_{G^c(P)}(x)$, in view of Theorem \ref{czdg}, to prove that $G^c(P)$ satisfies the Total Coloring Conjecture, it is enough to show that  $|[q_i]|< \frac{1}{4}  |V(G^c(P))|$ for some $i$.

		Suppose on the contrary that  $|[q_i]|\geq \frac{1}{4}  |V(G^c(P))|$ for all $i$. Let $|[q_i]|=l_i$ and   $|[q_i]^*| =m_i$, where $l_i, m_i\in \mathbb{N}$. Without loss of generality,   assume that $l_3\geq l_2\geq l_1$. By the assumption, $l_i\geq \frac{1}{4}  |V(G^c(P))|$ for every $i$ and $l_1+l_2+l_3+m_1+m_2+m_3=|V(G^c(P))|$. Hence $m_1+m_2+m_3 \leq \frac{1}{4}  |V(G^c(P))|$.
		
		\vskip 5truept 
		\underline{\underline{This implies that both $(m_1+m_3), m_2< \frac{1}{4}  |V(G^c(P))|\leq l_1, ~l_2$.    \hspace{.2in} --------------- (B)}}
		\vskip 5truept 
		
		Since $|[q_i]|=l_i$ and   $|[q_i]^*| =m_i$, we denote $[q_i]=\{q_{i1},q_{i2},\dots,q_{il_i}\}$ and $[q_i]^*=\{q_{i1}^*,q_{i2}^*,\dots,q_{im_i}^*\}$. By Lemma \ref{property}(4), it is clear that deg$(q_{ij})= $ deg$(q_{ik})$ for $q_{ij}, q_{ik} \in [q_i]$. By $(A)$, we  have $\Delta(G^c(P))=$deg$(q_{11}^*)=|V(G^c(P))|-l_1-1=l_1+l_2+l_3+m_1+m_2+m_3-l_1-1= l_2+l_3+m_1+m_2+m_3-1$. 
		
		\vskip 5truept 
		Now, we prove that $\chi''(G^c(P))\leq \Delta(G^c(P))+2= l_2+l_3+m_1+m_2+m_3+1$.
		\vskip 5truept 
		
		Let $G'$ be the vertex induced subgraph on the vertices  of $[q_1]^*,[q_2]^*,[q_3]^*$, and $[q_3]$. First, we do the total coloring of $G'$. Put $A=\{q_{11}^*,q_{12}^*,\dots,q_{1m_1}^*,~q_{21}^*,q_{22}^*,\dots,q_{2m_2}^*,~q_{31}^*,q_{32}^*,\dots,q_{3m_3}^*, q_{31},q_{32},\dots,q_{3k_3}\}$. Consider the complete graph $K_r$ on the set $A$, where $r=m_1+m_2+m_3+l_3$. By Theorem \ref{complete}, $\chi''(K_r)\leq \Delta(K_r)+2=m_1+m_2+m_3+l_3-1+2=m_1+m_2+m_3+l_3+1$. Thus, we can do the total  coloring of $K_r$ with at most $m_1+m_2+m_3+l_3+1$ colors.

		Now, we give the total coloring to $G'$. The vertex $x$ of $G'$ is colored by the color of the vertex $x$ given in the total coloring of $K_r$ and edge $xy$ of $G'$ is colored by the color of the edge $xy$ given in the total coloring of $K_r$. 

		Let's do the total coloring to the vertex induced subgraph by $[q_2]$. For this, the vertex $q_{2j}$ is colored by the color of the vertex $q_{3j}$ in the total coloring of $K_r$ and edge $q_{2j}q_{2k}$ is colored by the color of the edge $q_{3j}q_{3k}$ in the total coloring of $K_r$, where $j,k\in\{1,2,\dots,l_2\}\subseteq\{1,2,\dots,l_3\}$, as $l_2 \leq l_3$.
		
		Similarly, we do the total coloring to the vertex-induced subgraph by $[q_1]$.
		
		We do the edge coloring to the edges between the vertices of $[q_3]^*$ and the vertices of $[q_1]$. For this, the edge $q_{3j}^*q_{1k}$ is colored by the color of the edge $q_{3j}^*q_{3k}$  in the total coloring of $K_r$. Note that there are no edges between $[q_3]$ and $[q_3]^*$ in $G^c([P])$. Further, $l_3 \geq l_1$. Hence this edge coloring is possible.
		
		We denote the set $V_1=\{q_{21},q_{22},\dots,q_{2l_2}\}$ and $V_2=\{q_{11}^*,\dots,q_{1m_1}^*,~ q_{31}^*,\dots,q_{3m_3}^*\}$.  Consider the complete bipartite graph $K_{s,t}$ on  sets $V_1$ and $V_2$, where $s=l_2, ~ t=m_1+m_3$.
		
		By $(B)$ and  Theorem \ref{bipartite}, $\chi'(K_{s,t})=max\{s,t\}=s=l_2$. Note that we use $l_2$ different colors other than the color used in the total coloring of $K_r$ to do the edge coloring of $K_{r,s}$. Thus up till now, we have used $m_1+m_2+m_3+l_3+1$ colors for the total coloring of $K_r$ and $l_2$ colors to do the edge coloring of $K_{s,t}$.
		
		Put $V_1'=\{q_{11},q_{12},\dots,q_{1l_1}\}$ and $V_2'=\{q_{21}^*,\dots,q_{2m_2}^*\}$. We consider the complete bipartite graph $K'_{r',s'}$ on  sets $V_1'$ and $V_2'$, where $r'=l_1,~ s'=m_2$.
		
		Again by (B) and  Theorem \ref{bipartite}, $\chi'(K'_{r',s'})=max\{r',s'\}=r'=l_1$. To do the edge coloring of $K'_{r',s'}$ we choose $r'=l_1$ colors out of $l_2$ colors used in the edge coloring of $K_{s,t}$.
		
		Combing all, we get the proper total coloring of $G^c(P)$. It is not very difficult to prove that, this is the proper total coloring to $G^c(P)$ with $ l_2+l_3+m_1+m_2+m_3+1=\Delta(G^c(P))+2$ colors. Thus, $\chi''(G^c(P))\leq \Delta(G^c(P))+2$. Thus, in this case, $G^c(P)$ satisfies the Conjecture.

		\vskip 5truept

		\textbf{Case-III:} $n\geq4$.

		In view of the discussion in Case-II, if  $|[q_i]|< \frac{1}{4}  |V(G^c(P))|$ for some $i$, then we are through.
		
		Suppose on the contrary that  $|[q_i]|\geq \frac{1}{4}  |V(G^c(P))|$ for all $i$. Then  $|V(G^c(P))|\geq |[q_1]|+|[q_2]|+|[q_3]|+|[q_4]|+|[q_1]^*|+|[q_2]^*|+|[q_3]^*|+|[q_4]^*|\geq \frac{1}{4}  |V(G^c(P))|+\frac{1}{4}  |V(G^c(P))|+\frac{1}{4}  |V(G^c(P))|+\frac{1}{4}  |V(G^c(P))|+1+1+1+1$, as $|[q_i]^*| \geq 1$. This implies that $|V(G^c(P))|\geq |V(G^c(P))|+4$, a contradiction. Thus $|[q_i]|< \frac{1}{4}  |V(G^c(P))|$ for some $i$. Therefore $G^c(P)$ satisfies the Conjecture in this case too. 
		
		This completes the proof.	\end{proof}


	\begin{remark}
		
		Now, we show that the techniques used in Theorem \ref{zdgtcc} can not be used to prove Theorem \ref{czdgtcc} and vice versa. 	
		
		It is easy to see that (in Figure $\ref{example} (b)$) $\Delta(G(P))=12$ and $|V(G(P))|=24$. Clearly, $\Delta(G(P))\ngeq\frac{3}{4} |V(G(P))|$. However by Theorem \ref{zdgtcc}, $G(P)$ satisfies the Total Coloring Conjecture.
		
		From Figure $\ref{example} (d)$,  $|V(G^c(P))|=24$ and $\Delta(G^c(P))=19$. Further,  $\alpha(G^c(P))=3$. It is easy to verify that, $|S|\ngeq |V(G^c(P))| - \Delta(G^c(P))-1$ for any independent set $S$ of $G^c(P)$. However, by Theorem \ref{czdgtcc}, $G^c(P)$ satisfies the Conjecture.
	\end{remark}

	\begin{remark}\label{tccrmk}
		It is easy to observe that if $G$ satisfies the Total Coloring Conjecture, then $G+ I_m$ satisfies the Total Coloring Conjecture. In view of Theorem \ref{czdg}, for any finite simple graph $G$, $G\vee K_m$ satisfies the Total Coloring Conjecture, where $m\in \mathbb{N}$.
		
	\end{remark}

	\section{Applications}\label{applications}
	In this section, we study the interplay of the zero-divisor graphs of ordered sets and the graphs associated with algebraic structures. Also,
	We apply our results to various graphs associated with algebraic structures such as the comaximal ideal graph of rings, the (co-)annihilating ideal graph of rings,  the zero-divisor graph of reduced rings, the intersection graphs of rings, and the intersection graphs of subgroups of groups.
	
	\vskip10pt 
	\textbf{I. The comaximal ideal graph of a ring  }
	
	\vskip10pt

	Let $R$ be a commutative ring with identity. Then $(Id(R),\leq)$  is a modular, $1$-distributive lattice (cf. discussion before Proposition 1.10 of Atiyah and MacDonald \cite{AM} and the fact that the ideal multiplication distributes over ideal sum) under the set inclusion as a partial order. Clearly, sup$\{I,J\}=I+J$ and  inf$\{ I, J\} = I\cap J$.
	It is well known that the lattice $Id( R)$ is a complete lattice with the ideals $(0)$ and $R$ as its least and the greatest element, respectively. Henceforth, we denoted the lattice $Id(R)$ by $L$. Let $L^\partial$ be the dual of the lattice of $L$. Therefore in $L^\partial$, sup$_{L^\partial}\{I, J\} = I\cap J$ and
	inf$_{L^\partial}\{I, J\} = I+ J$. The ideal $R$ is the least element of $L^\partial$, and the ideal $(0)$ is the greatest element of $L^\partial$. Since modularity is a self dual notion, $L^\partial$ is also modular. Further, by the duality, $   L^\partial$ is a $0$-distributive lattice. Moreover, the maximal ideals of $R$ are nothing but the atoms of $L^\partial$. Therefore, $L^\partial$ is an atomic lattice. With this preparation, we prove that the comaximal ideal graph of $R$ is nothing but the zero-divisor graph of $L^\partial$.

	\begin{definition}{[Ye and Wu \cite{yewu}]}
		Let $R$ be a commutative ring with identity. We associate a simple undirected graph with $R$, called as the {\it comaximal ideal graph $\mathbb{CG}(R)$} of $R$, where the vertices of $\mathbb{\mathbb{CG}}(R)$ are proper ideals of $R$ that are not contained in the Jacobson radical $J(R)$ of $R$ and two vertices $I$ and $J$ are adjacent if and only if $I+J=R$.
		
	\end{definition}

	The following is the modified definition of the comaximal ideal graph. 
	The \textit{modified comaximal ideal graph} $\mathbb{CG}^*(R)$ is the graph with the vertex set as  the nonzero proper ideals of $R$ and two
	vertices $I$ and $J$ are adjacent if and only if $I$ and $J$ are comaximal. 
	Clearly, one can see that  the  comaximal ideal graph $\mathbb{CG}(R)$ is the 
	subgraph of the modified comaximal ideal graph  $\mathbb{CG}^*(R)$. Moreover, $\mathbb{CG}^*(R)=\mathbb{CG}(R)+I_m$, where $m=|J(R)|-1$.

	\vskip 5truept 
	
	\begin{theorem}\label{same}
		Let $R$ be a  commutative ring with identity and let $L^\partial$ be the dual of the lattice $Id (R)$ of all ideals of $R$. Then $\mathbb{CG}(R)=G(L^\partial)$.
	\end{theorem}
	
	\begin{proof} 
		It is clear that if a ring is a local ring, in this case, the comaximal ideal graph is a null graph. Hence we consider non-local commutative rings only. Further, in the case of zero-divisor graphs of lattices, it is clear that if a lattice has exactly one atom, then its zero-divisor graph is a null graph. 
		
		Hence, now consider that $R$ is non-local.
		
		Observe that  $V(G(L^\partial)) =\{ I\in L^\partial\setminus \{0\}| I\wedge J=0$ for some $J\in L^\partial\setminus \{0\}\}$, where $0$ is the least element of $L^\partial$ which is $R$. Hence $V( G(L^\partial) ) =\{ I\in Id(R)\setminus \{R\}| I+ J=R$ for some $J\in Id(R)\setminus \{R\} \}$.  Further, $I$ is adjacent to $ J$ in $G(L^\partial)$ if and only if $\inf_{L^\partial}\{I, J\}=\{0\}$, that is, $I+ J=R$. 
		\par 	Now, we prove $V(\mathbb{CG}(R))=V(G(L^\partial))$. For this, let $I\in V(\mathbb{CG}(R))$. Hence $I\in Id(R)\setminus\{R\}$ and $ I\nsubseteq J(R)$. Therefore  there exists a maximal ideal $M_i$ such that $I\nsubseteq M_i$. Clearly, $M_i \not\subseteq J(R)$ and $I+M_i=R$. Hence $I\in V(G(L^\partial))$. Thus, $V(\mathbb{CG}(R))\subseteq V(G(L^\partial))$. 
		
		\par  Let $I\in V(G(L^\partial))$. Hence $ I\in Id(R)\setminus \{R\}$ and there exists $J\neq R$  such that $I+ J=R$. We have to prove that $I\nsubseteq J(R)$. Suppose on the contrary that $I\subseteq J(R)$. Therefore $I\subseteq M_i$ for all $i$. Since $J\neq R$, then $J\subseteq M_i$ for some $i$. Hence  $R=I+J\subseteq M_i$, a contradiction. Hence $I\nsubseteq J(R)$. This proves that $I\in V(\mathbb{CG}(R))$ and hence $V(G(L^\partial))\subseteq V(\mathbb{CG}(R))$.   		 Further, adjacency in $\mathbb{CG}(R)$ and $G(L^\partial)$ is same.  
		
		Thus, the comaximal ideal graph of $R$ is the same as the zero-divisor graph of lattice $L^\partial$. \end{proof}
	
	\vskip 5truept

	\begin{remark} In view of Theorem \ref{same}, it is clear that the study of the comaximal ideal graph of a commutative ring $R$ with identity is nothing but the study of the zero-divisor graph of a $0$-distributive, modular lattice.\end{remark}
	
	\begin{remark}
		We rewrite the vertex set of the comaximal ideal graph $\mathbb{CG}(R)$  of a commutative ring $R$ with identity as follows:
		$V(\mathbb{CG}(R))=\{ I\in Id(R)\setminus \{R\} ~ | ~I+ J=R$ for some $J\neq R\}$. 
	\end{remark}
	
	\vskip 5truept 
	
	Let $R$ be an Artinian ring. Then $R=\prod\limits_{i=1}^nR_i$, where $R_i$ is an Artinian local ring for every $i$. It is easy to observe that $R$ has identity if and only if $R_i$ has identity for all $i$. Hence it follows from the result of Chajda and L\"anger \cite[Theorem 2]{cl} that  $Id(R)= Id(R_1) \times Id(R_2) \times  \cdots \times Id(R_n) $, where $Id(R_1) \times Id(R_2)  \times \cdots \times Id(R_n) $ denotes the (lattice) direct product of the ideal lattices $Id(R_i)$ of $R_i$. Hence, $Id(R)^\partial=  Id(R_1)^\partial \times Id(R_2)^\partial \times \times \cdots \times Id(R_n)^\partial $.  
	
	\vskip 5truept 
	
	The following result is immediate from Theorems \ref{edge chromatic number},  \ref{zdgtcc},  \ref{czdgtcc}, Remark \ref{type2} and Theorem \ref{same}.

	\vskip 5truept 
	
	\begin{corollary}\label{cgtcc}
		Let $R$ be a commutative ring with finitely many ideals and let $\mathbb{CG}(R)$ be its comaximal ideal graph. Then $\chi'(\mathbb{CG}(R))=\Delta(\mathbb{CG}(R))$ and $\mathbb{CG}(R)$, $\mathbb{CG}^c(R)$ satisfies the Total Coloring Conjecture. Moreover, $\mathbb{CG}(R)$ is of type II if and only if $R=R_1\times R_2$ with $|Id(R_1)|=|Id(R_2)|$. 
	\end{corollary}

	\vskip 5truept

	Now, we characterize the chordal comaximal ideal graphs of Artinian rings.

	\vskip 5truept 
	As an immediate consequence of  Corollary \ref{zdgchordproduct} and Theorem \ref{same}, we have the following result.

	\begin{corollary} \label{cgchord}
		Let $R$ be an Artinian ring with finitely many ideals. Then
		
		\textbf{(A)} $\mathbb{CG}(R)$ is chordal if and only if one of the following hold:
		
		\begin{enumerate}
			\item $R$ is a local ring;
			
			\item  $R=R_1\times R_2$ and $R_i$ is field for some $i\in \{1,2\}$;

			\item  $R=F_1\times F_2\times F_3$, where $F_i$ is field for all $i\in \{1,2,3\}$.
		\end{enumerate}
		
		\textbf{(B)} $\mathbb{CG}^{c}(R)$ is chordal if and only if the number of maximal ideals of $R$ is at most $3$.	
		
	\end{corollary}

	\vskip 5truept

	The following result is immediate from  Corollary \ref{perfect0lattice} and  Theorem \ref{same}.
	
	\begin{corollary}\label{cgperfect}
		Let $R$ be an Artinian ring with finitely many ideals. Then $\mathbb{CG}(R)$ is  perfect if and only if $\omega(\mathbb{CG}(R))=|\text{Max}(R)|\leq 4$.
	\end{corollary}
	
	This extends the following result.
	
	\begin{corollary}  [{\cite[Theorem 3.6]{azadi}}] Let $R$ be an Artinian ring with finitely many ideals.
		If $Max(R)\leq 4$, then $\mathbb{CG}(R)$ is a perfect graph.
	\end{corollary}

	\vskip 15truept

	\noindent \textbf{II. Annihilating ideal graph and  co-annihilating ideal graph of ring}
	
	\vskip 5truept 
	Many researchers studied annihilating ideal graphs and co-annihilating ideal graphs of rings. In this section, we characterize perfect annihilating ideal graphs and co-annihilating ideal graphs of rings.
	\vskip 5truept
	
	\begin{definition}[Akbari et al. \cite{akbari}]
		The \textit{co-annihilating-ideal graph} of $R$, denoted by $\mathbb{CAG}^*(R)$, is a graph whose vertex set is
		the set of all non-zero proper ideals of $R$ and two distinct vertices $I$ and $J$ are adjacent
		whenever Ann$(I)\cap$ Ann$(J) =(0)$.
	\end{definition}

	\begin{definition}[Visweswaran and Patel \cite{viswe}]
		Let $R$ be a commutative ring with identity. Associate a simple undirected graph with $R$, called as the {\it annihilating ideal graph $\mathbb{AG}(R)$} of $R$, where the vertices of $\mathbb{AG}(R)$ are nonzero annihilating ideals of $R$, and two vertices $I$ and $J$ are adjacent if and only if $I+J$ is also an annihilating ideal of $R$.
	\end{definition}
	
	The following is the modified definition of annihilating ideal graph. The modified \textit{annihilating ideal graph} $\mathbb{AG}^*(R)$ be a simple undirected graph with vertex set $V(\mathbb{AG}^*(R))$ is set of all nonzero proper ideals of $R$ and two distinct vertices $I, J$ are joined if and only if $I+J$ is also an annihilating ideal of $R$.
	
	\begin{lemma}[{\cite[Lemma 2.1]{aghapouramin}}]
		$I-J$ is an edge of $\mathbb{AG}^*(R)$ if and only if Ann$(I) ~\cap$ Ann$(J)\neq(0)$. Hence $\mathbb{AG}^{*c}(R)=\mathbb{CAG}^*(R)$.
	\end{lemma}

	\begin{observation} \label{obs3}
		If $R$ is an Artinian ring, then for any non-zero proper ideal $I$ of $R$,
		Ann$_R(I)\neq (0)$. Hence in  $R$, if  non-zero proper ideals $I$ and $J$ are adjacent in $\mathbb{CAG}^*(R)$ if and only if  Ann$(I)\cap$ Ann$(J) =$   Ann$(I+J) =(0)=$Ann$(R)$. Thus, $I$ and $J$ are adjacent in $\mathbb{CAG}^*(R)$ if and only if $I+J=R$, whenever $R$ is an Artinian ring.
	\end{observation}
	
	\begin{lemma}[{\cite[Corollary  1.2]{akbari}}]
		\label{artiniansame1}
		If $R$ is an Artinian ring, then $\mathbb{CG}^*(R)=\mathbb{CAG}^*(R)$.
	\end{lemma}
	
	Hence, we have the following result.
	
	\begin{theorem}\label{artiniansame}
		If $R$ is an Artinian ring, then $\mathbb{CG}^*(R)=\mathbb{CG}(R)+I_m=\mathbb{CAG}^*(R)=\mathbb{AG}^{*c}(R)=\mathbb{AG}^{c}(R)$,  where $m=|J(R)|-1$.
	\end{theorem}

	Hence the following result is immediate from  Remark \ref{tccrmk}, Corollary \ref{cgtcc}, and Theorem   \ref{artiniansame}.
	
	\begin{corollary}\label{cagtcc}
		Let $R$ be a commutative ring with finitely many ideals and let $\mathbb{CAG}^*(R), ~ \mathbb{AG}^*(R)$ be its co-annihilating ideal graph and annihilating ideal graph respectively. Then $\chi'(\mathbb{CAG}^*(R))=\chi'(\mathbb{AG}^{*c}(R))=\Delta(\mathbb{CAG}^*(R))=\Delta(\mathbb{AG}^{*c}(R))$ and $\mathbb{CAG}^*(R)=\mathbb{AG}^{*c}(R)$, $\mathbb{CAG}^{*c}(R)=\mathbb{AG}^{*}(R)$ satisfies the Total Coloring Conjecture. Moreover, $\mathbb{CAG}^*(R)=\mathbb{AG}^{*c}(R)$ is of type II if and only if $R=R_1\times R_2$ with $|Id(R_1)|=|Id(R_2)|$. 
	\end{corollary}

	The following result is immediate from  Remark \ref{obs1},  Corollary \ref{cgchord} and Theorem \ref{artiniansame}.

	\begin{corollary}\label{cagchord} 
		Let $R$ be an Artinian ring with finitely many ideals. Then
		
		\textbf{(A)} $\mathbb{CAG}^*(R)=\mathbb{AG}^{*c}(R)$ is chordal if and only if one of the following hold:
		
		\begin{enumerate}
			\item $R$ is local ring;
			
			\item  $R=R_1\times R_2$ and $R_i$ is field for some $i\in \{1,2\}$;

			\item  $R=F_1\times F_2\times F_3$, where $F_i$ is field for all $i\in \{1,2,3\}$.
		\end{enumerate}
		
		\textbf{(B)} $\mathbb{CAG}^{*c}(R)=\mathbb{AG}^*(R)$ is chordal if and only if the number of maximal ideals of $R$ is at most $3$.	
		
	\end{corollary}

	Hence the following result is immediate from   Remark \ref{obs2}, Corollary \ref{cgperfect}, and Theorem \ref{artiniansame}. The perfectness   of $\mathbb{AG}(R)$
	is studied  in \cite[Corollary 
	2.3]{aghapouramin}. The Corollary \ref{cagchord} and Corollary \ref{cagperfect} are essentially proved for co-annihilating ideal graph in \cite{mirghadim}.
	
	\begin{corollary}\label{cagperfect}
		Let $R$ be an Artinian ring with finitely many ideals. Then $\mathbb{CAG}^*(R)$ is perfect if and only if $\mathbb{AG}^*(R)$ is perfect if and only if $\omega(\mathbb{CG}^*(R))=$ number of maximal ideals of $R$ is at most 4.
	\end{corollary}

	\vskip 15truept

	\noindent \textbf{III. The zero-divisor graph of a reduced ring  }
	
	\vskip10pt 
	In \cite{djl}, it is mentioned that if $R$ is a  reduced Artinian ring with exactly $k$ prime ideals, then there exist fields $F_1,\dots,F_k$ such that $R\cong F_1\times\cdots\times F_k$. Further, it is  proved that the ring-theoretic zero-divisor graph   $\Gamma(R)=\Gamma(\prod\limits_{i=1}^kF_i)$ equals the poset-theoretic zero-divisor graph of $R$ ($R$ treated  as a poset under the partial order given in \cite[Lemma 3.3]{djl}), that is,  $\Gamma(\prod\limits_{i=1}^kF_i)=G(\prod\limits_{i=1}^kF_i)$ (cf. \cite[Remark 3.4]{LR1}). 
	
	Hence the following result is immediate from Theorems \ref{tcc},  \ref{edge chromatic number}, Remark \ref{type2},    Theorems \ref{zdgtcc}, \ref{czdgtcc}.

	The first part of the result, i.e., $\chi'(\Gamma(R))=\Delta(\Gamma(R))$, $\Gamma(R)$ satisfies Total Coloring Conjecture and last part of result, i.e., $\Gamma(R)$ is of type II if and only if $R=F_1\times F_2$ with $|F_1|=|F_2|$ are proved in \cite{nkvj}.

	\begin{corollary}
		Let $R$ be a finite reduced ring and let $\Gamma(R)$ be its ring-theoretic zero-divisor graph. Then $\chi'(\Gamma(R))=\Delta(\Gamma(R))$ and $\Gamma(R)$, $\Gamma^c(R)$ satisfies the Total Coloring Conjecture. Moreover, $\Gamma(R)$ is of type II if and only if $R=F_1\times F_2$ with $|F_1|=|F_2|$. 
	\end{corollary}

	In view of Corollary \ref{zdgchordproduct}, we have the following result.

	\begin{corollary} \label{zdgchordfinite}
		Let $R$ is a  finite reduced ring. Then
		
		\textbf{(A)} $\Gamma(R)$ is chordal if and only if one of the following holds:
		
		\begin{enumerate}
			\item $R$ is a field;
			
			\item  $R=F_1\times F_2$ and $|F_i|=2$ for some $i\in \{1,2\}$, i.e., $F_i\cong \mathbb{Z}_2$ for some $i\in \{1,2\}$;

			\item  $R=F_1\times F_2\times F_3$ and $|F_i|=2$ for all $i\in \{1,2,3\}$, i.e., $R\cong \mathbb{Z}_2\times \mathbb{Z}_2\times \mathbb{Z}_2$.
		\end{enumerate}
		
		\textbf{(B)} $\Gamma^c(R)$ is chordal if and only if number of maximal ideals $n\leq 3$, i.e., $R\cong F_1\times F_2 \times F_3$.	
		
	\end{corollary}

	In view of Corollary \ref{zdgperfectproduct}, we have the following result.

	\begin{corollary}[Smith \cite{smith}] \label{zdgperfectfinite}
		Let $R$ is a finite reduced  ring. Then $\Gamma(R)$ is perfect if and only if the number of maximal ideals of $R$ is at most  $ 4$, i.e., $R\cong F_1\times F_2 \times F_3 \times F_4$.
	\end{corollary}

	\vskip 5truept

	It is not very difficult to prove that, $P$ be a finite bounded poset has a unique dual atom if and only if   $Z(P^\partial)=\{0\}$, where $P^\partial$ is dual of $P$. If $\textbf{P}=\prod \limits _{i=1}^{n}P^i$, where $P^i$ is a finite bounded poset for all $i$, then it is easy to see that, $\textbf{P}^\partial=\prod \limits _{i=1}^{n}(P^i)^\partial$. If a finite poset $P$ has exactly one atom, then $\{x,y\}^\ell=\{0\}$, where $0$ is the least element of $P$ (the lower cone taken in $P$), if and only if  either $x=0$ or $y=0$, for all $x, y\in P$. Also, if a finite poset $P$ has exactly one dual atom, then $\{x,y\}^\ell=\{0\}$, where $0$ is the least element  of $P^\partial$ (the lower cone taken $P^\partial$), if and only if  either $x=0$ or $y=0$, for all $x,y\in P^\partial$. 
	
	\begin{lemma}\label{5.12}
		Let $\textbf{P}=\prod \limits _{i=1}^{n}P^i$, where $P^i$ is a finite bounded poset that has exactly one atom and exactly one dual atom for all $i$. Then the zero-divisor graph of $\textbf{P}$ is equal to the zero-divisor graph of $\textbf{P}^\partial$.
	\end{lemma}

	If $C$ is a finite chain, then $\{x,y\}^\ell=\{0\}$, if and only if  either $x=0$ or $y=0$, for all $x,y\in C$. Further, every finite chain has exactly one atom and exactly one dual atom. From this fact and Lemma \ref{5.12},  the following result is immediate. 
	
	\begin{lemma}\label{chains}
		The zero-divisor graph of  $\textbf{P}=\prod\limits_{i=1}^nP^i$, where $P^i$ is a finite bounded poset with $Z(P^i)=\{0\}$  for all  $i$ is isomorphic to the zero-divisor graph of  $\textbf{C}=\prod\limits_{i=1}^nC_i$, where $|C_i|=|P^i|$ for all $i$.  Moreover, the zero-divisor graph of $\textbf{C}$ is equal to the zero-divisor graph of $\textbf{C}^\partial$.
		
	\end{lemma}
	
	\begin{remark}
		Note that though the zero-divisor graph of $\textbf{C}$ is equal to the zero-divisor graph of $\textbf{C}^\partial$, the zero-divisor graph of $\textbf{P}^\partial$ where  $\textbf{P}=\prod\limits_{i=1}^nP^i$, with $P^i$ is a finite bounded poset such that  $Z(P^i)=\{0\}$ for every $i$, is not isomorphic to the zero-divisor graph of $\textbf{C}^\partial$. However, each $P^i$ has a unique dual atom, then $G(\textbf{C}^\partial) \cong G(\textbf{P}^\partial)$. 
	\end{remark}

	\vskip 15truept

	\noindent 	\textbf{IV. Intersection graphs of ideals of  Artinian principal ideal rings   }
	
	\vskip10pt 
	
	Let $R$ be a  commutative ring with identity. Then the \textit{intersection graph of ideals} $\mathbb{IG}(R)$ of $R$ is the graph whose vertices are the nonzero proper   ideals of $R$ such that distinct vertices $I$ and $J$ are adjacent if and only if $I\cap J\neq\{0\}$; see \cite{chakra}. The vertex set of the  zero-divisor graph  ${G^*}(Id(R))$ of a lattice $Id(R)$ is $Id(R)\setminus \{0_{_{Id(R)} },1_{_{Id(R)}}\}$, that is, the set of nonzero proper   ideals of $R$. Therefore $V(\mathbb{IG}(R))=V({G^*}(Id(R)))=V({G}^{*c}(Id(R)))$.  The ideals $I$ and $J$ are adjacent in ${G^*}(Id(R))$ if and only if $I\wedge J=0_{_{Id(R)}}$, that is, the ideals $I$ and $J$ are adjacent in ${G^*}(Id(R))$ if and only if $I\cap J=(0)$. The ideals $I$ and $J$ are adjacent in ${G}^{*c}(Id(R))$ if and only if $I\cap J\neq (0)$. Hence  the following result.
	
	\begin{lemma}\label{igczdg} Let $R$ be a  commutative ring with identity. Then  $\mathbb{IG}(R)={G}^{*c}(Id(R))$.
	\end{lemma}

	The following discussion can be found in \cite{djl}.
	
	Let $R$ be a commutative ring with identity. Recall that $R$ is a \emph{special principal ideal ring} (or, \emph{SPIR} for brevity) if $R$ is a local Artinian principal ideal ring (cf. \cite{Hung}). If $R$ is an SPIR with maximal ideal $M$, then there exists $n\in\mathbb{N}$ such that $M^n=\{0\}$, $M^{n-1}\neq\{0\}$, and if $I$ is an ideal of $R$, then $I=M^i$ for some $i\in\{0,1,\dots,n\}$ (\cite[Proposition 4]{Hung}). In this case, $M$ is nilpotent with index of nilpotency equal to $n$, and the lattice of ideals of $R$ is isomorphic to the chain $C$ of length $n$ (chain of $n+1$ elements). 
	
	By \cite[Lemma 10]{Hung}, $R$ is an Artinian principal ideal ring if and only if there exist SPIRs $R_1,\dots,R_n$ such that $R\cong R_1\times\cdots\times R_n$ (it is also a straightforward consequence of the structure theorem of Artinian rings in \cite[Theorem 8.7]{AM}). Thus, $Id(R)$ of Artinian principal ideal ring is product of chains $\textbf{C}=\prod\limits_{i=1}^nC_i$  ($C_i$ is chain of length $n_i$), where $R\cong R_1\times\cdots\times R_n$, nilpotency index of maximal ideal $M_i$ of $R_i$ is $n_i$. 
	
	\vskip 5truept 
	
	By Theorems \ref{edge chromatic number},  \ref{zdgtcc},  \ref{czdgtcc}, Remark \ref{type2}, Lemma \ref{chains} and Lemma \ref{igczdg}, the following result  follows.

	\vskip 5truept 
	
	\begin{corollary}\label{spirtcc} 
		Let $R$ be an Artinian principal ideal ring and let $\mathbb{IG}(R)$ be the intersection graph of ideals of $R$.  Then $\chi'(\mathbb{IG}^c(R))=\Delta(\mathbb{IG}^c(R))$ and $\mathbb{IG}(R)$, $\mathbb{IG}^c(R)$ satisfies the Total Coloring Conjecture. Moreover, $\mathbb{IG}^c(R)$ is of type II if and only if $R=R_1\times R_2$ with $|Id(R_1)|=|Id(R_2)|$.
	\end{corollary}
	
	\noindent 	The next result characterizes chordal graphs $\mathbb{IG}^c(R)$ and $\mathbb{IG}(R)$ for an Artinian principal ideal ring $R$.	
	
	\vskip 5truept 
	
	The  Corollary \ref{zdgchordproduct}, Lemma \ref{chains} and Lemma \ref{igczdg}, yields the following result.

	\vskip 5truept 
	
	\begin{corollary}\label{spirchordal} \label{igchord}
		Let $R$ be an Artinian principal ideal ring. Then
		
		\textbf{(A)} $\mathbb{IG}^c(R)$ is chordal if and only if one of the following holds:
		
		\begin{enumerate}
			\item $R$ is a local ring;
			
			\item  $R=R_1\times R_2$ and $R_i$ is a field for some $i\in \{1,2\}$;

			\item  $R=F_1\times F_2\times F_3$, where $F_i$ is a field for all $i\in \{1,2,3\}$.
		\end{enumerate}
		
		\textbf{(B)} $\mathbb{IG}(R)$ is chordal if and only if number of maximal ideals of $R$ is at most $ 3$.	
		
	\end{corollary} 
	
	\vskip 5truept 
	
	From Corollary \ref{zdgperfectproduct}, Lemma \ref{chains} and Lemma \ref{igczdg}, we have:

	\vskip 5truept

	\begin{corollary}\label{spirperfect}
		Let $R$ be an Artinian principal ideal ring. Then $\mathbb{IG}(R)$ is perfect if and only if the number of maximal ideals of $R$ is at most $ 4$.
	\end{corollary}
	
	\begin{theorem}[Das \cite{das}]
		The intersection graph of ideals of $Z_n$ is perfect if and only if
		$n = p_1^{n_1}p_2^{n_2}p_3^{n_3}p_4^{n_4}$, where $p_i$ are distinct primes and $n_i\in \mathbb{N}\cup \{0\}$, that is, the
		number of distinct prime factors of $n$ is less than or equal to $4$.  
	\end{theorem}

	\vskip 5truept 
	
	\noindent \textbf{  Note that the Corollary \ref{cgchordperfect} follows from Corollary \ref{cgchord}, \ref{cgperfect}, Theorem \ref{artiniansame}, Corollary \ref{cagtcc}, \ref{cagchord}, \ref{cagperfect}, \ref{spirtcc}, \ref{spirchordal}, \ref{spirperfect}.}
	
	\vskip15pt

	\noindent 	\textbf{V. Intersection graph of subgroups of a  group   }
	
	\vskip10pt 
	
	Let $\mathcal{G}$ be a  group. Then the \textit{intersection graph of subgroups} $\mathbb{IG}(\mathcal{G})$ of $\mathcal{G}$ is the graph whose vertices are the proper  non-trivial subgroups of $\mathcal{G}$ such that distinct vertices $H$ and $K$ are adjacent if and only if $H\cap K\neq\{e\}$. Let $L(\mathcal{G})$ be the subgroup lattice of  $\mathcal{G}$. The vertex set of ${G^*}(L(\mathcal{G}))$ is $L(\mathcal{G})\setminus \{0_{L(\mathcal{G})},1_{L(\mathcal{G})}\}$, that is, the set of proper  non-trivial subgroups of $\mathcal{G}$. Therefore, $V(\mathbb{IG}(\mathcal{G}))=V({G^*}(L(\mathcal{G})))=V(\mathbb{G}^{*c}(L(\mathcal{G})))$.  The subgroups $H$ and $K$ are adjacent in ${G}(L(\mathcal{G}))$ if and only if $H\wedge K=0_{L(\mathcal{G})}$, that is, the subgroups $H$ and $K$ are adjacent in ${G}(L(\mathcal{G}))$ if and only if $H\cap K=(e)$. The subgroups $H$ and $K$ are adjacent in ${G}^{*c}(L(\mathcal{G}))$ if and only if $I\cap J\neq (e)$. Hence we have the following result.
	
	\begin{lemma}
		Let $\mathcal{G}$ be a  group and  $L(\mathcal{G})$ be the subgroup lattice of $\mathcal{G}$. Then	$\mathbb{IG}(\mathcal{G})={G}^{*c}(L(\mathcal{G}))$.
		
	\end{lemma}

	\begin{remark}\label{groupspir} If $p_1,\dots,p_k\in\mathbb{N}$ are (not necessarily distinct) prime numbers, then the intersection graph of subgroups  $\mathbb{IG}(\mathbb{Z}_{p_1^{n_1}}\times\cdots\times\mathbb{Z}_{p_k^{n_k}})$ of the group  $\mathbb{Z}_{p_1^{n_1}}\times\cdots\times\mathbb{Z}_{p_k^{n_k}}$ is same as the intersection graph of ideals of the ring $\mathbb{Z}_{p_1^{n_1}}\times\cdots\times\mathbb{Z}_{p_k^{n_k}}$. Since $\mathbb{Z}_{p_i^{n_i}}$ is an SPIR whose maximal ideal has index of nilpotency equal to $n_i$ for every $i\in\{1,\dots,k\}$.
		
	\end{remark}	
	
	\vskip 5truept 
	In view of Corollaries \ref{spirtcc}, \ref{spirchordal}, \ref{spirperfect} and Remark \ref{groupspir}, we have the following Corollaries. 
	
	\vskip 5truept 
	
	\begin{corollary} 
		Let $\mathcal{G}=\mathbb{Z}_{p_1^{n_1}}\times\cdots\times\mathbb{Z}_{p_k^{n_k}}$ be a group, where $p_1,\dots,p_k\in\mathbb{N}$ are (not necessarily distinct) prime numbers and let $\mathbb{IG}(\mathcal{G})$ be the intersection graph of ideals of $\mathcal{G}$.  Then $\chi'(\mathbb{IG}^c(\mathcal{G}))=\Delta(\mathbb{IG}^c(\mathcal{G}))$ and $\mathbb{IG}(\mathcal{G})$, $\mathbb{IG}^c(\mathcal{G})$ satisfies the Total Coloring Conjecture. Moreover, $\mathbb{IG}^c(\mathcal{G})$ is of type II if and only if $R=R_1\times R_2$ with $|Id(R_1)|=|Id(R_2)|$.
	\end{corollary}

	\begin{corollary} \label{iggchord}
		Let $\mathcal{G}=\mathbb{Z}_{p_1^{n_1}}\times\cdots\times\mathbb{Z}_{p_k^{n_k}}$ be a group, where $p_1,\dots,p_k\in\mathbb{N}$ are (not necessarily distinct) prime numbers. Then
		
		\textbf{(A)} $\mathbb{IG}^c(\mathcal{G})$ is chordal if and only if one of the following hold:
		
		\begin{enumerate}
			\item $\mathcal{G}$ is a cyclic group of order $p_1^{n_1}$;
			
			\item  $\mathcal{G}=\mathbb{Z}_{p_1^{n_1}}\times \mathbb{Z}_{p_2^{n_2}}$, with $n_i=1$ for some $i\in \{1,2\}$;

			\item  $\mathcal{G}=\mathbb{Z}_{p_1^{n_1}}\times \mathbb{Z}_{p_2^{n_2}}\times \mathbb{Z}_{p_3^{n_3}}$, with $n_i=1$ for all $i\in \{1,2,3\}$.
		\end{enumerate}
		
		\textbf{(B)} $\mathbb{IG}(\mathcal{G})$ is chordal if and only if number of maximal subgroups of $\mathcal{G}$ is at most $3$ (that is, $k\leq 3$).	
	\end{corollary}

	\begin{corollary}
		Let $\mathcal{G}=\mathbb{Z}_{p_1^{n_1}}\times\cdots\times\mathbb{Z}_{p_k^{n_k}}$ be a group, where $p_1,\dots,p_k\in\mathbb{N}$ are (not necessarily distinct) prime numbers. Then $\mathbb{IG}(\mathcal{G})$ is  perfect if and only if  the number of maximal subgroups of $\mathcal{G}$ is at most $4$ (that is, $k\leq 4$).
	\end{corollary}

	%

	\vskip 5truept

	\noindent \textbf{Acknowledgment:} The first author is financially supported by the Council of Scientific and Industrial Research(CSIR), New Delhi, via Junior Research Fellowship Award Letter No. 09/137(0620)/2019-EMR-I.

	\vskip 10pt 
	\noindent\textbf{Conflict of interest:} The authors declare that there is no conflict of interest regarding the publishing of this paper.
	\vskip 10pt 
	\noindent\textbf{Authorship Contributions:} Both the authors contributed equally in the study of zero-divisor graphs of ordered sets and their applications to algebraic structures. Both the authors read and approved the final version of the manuscript.
	
	\vskip10pt 

\end{document}